\documentclass[reqno,10pt]{amsart}

\usepackage[a4paper,left=35mm,right=35mm,top=30mm,bottom=30mm,marginpar=25mm]{geometry} 
\usepackage{amsmath}
\usepackage{amssymb}
\usepackage{amsthm}
\usepackage{amscd}
\usepackage[ansinew]{inputenc}
\usepackage{color}
\usepackage[final]{graphicx}

\renewcommand{\epsilon}{\varepsilon}

\numberwithin{equation}{section}

\newtheoremstyle{thmlemcorr}{10pt}{10pt}{\itshape}{}{\bfseries}{.}{10pt}{{\thmname{#1}\thmnumber{ #2}\thmnote{ (#3)}}}
\newtheoremstyle{thmlemcorr*}{10pt}{10pt}{\itshape}{}{\bfseries}{.}\newline{{\thmname{#1}\thmnumber{ #2}\thmnote{ (#3)}}}
\newtheoremstyle{defi}{10pt}{10pt}{\itshape}{}{\bfseries}{.}{10pt}{{\thmname{#1}\thmnumber{ #2}\thmnote{ (#3)}}}
\newtheoremstyle{remexample}{10pt}{10pt}{}{}{\bfseries}{.}{10pt}{{\thmname{#1}\thmnumber{ #2}\thmnote{ (#3)}}}
\newtheoremstyle{ass}{10pt}{10pt}{}{}{\bfseries}{.}{10pt}{{\thmname{#1}\thmnumber{ A#2}\thmnote{ (#3)}}}

\theoremstyle{thmlemcorr}
\newtheorem{theorem}{Theorem}
\numberwithin{theorem}{section}
\newtheorem{lemma}[theorem]{Lemma}

\newtheorem{proposition}[theorem]{Proposition}

\theoremstyle{thmlemcorr*}
\newtheorem{theorem*}{Theorem}
\newtheorem{lemma*}[theorem]{Lemma}
\newtheorem{corollary*}[theorem]{Corollary}
\newtheorem{proposition*}[theorem]{Proposition}
\newtheorem{problem*}[theorem]{Problem}
\newtheorem{conjecture*}[theorem]{Conjecture}
\newtheorem{theoremconjecture*}[theorem]{Theorem/Conjecture}

\theoremstyle{defi}

\theoremstyle{remexample}
\newtheorem{remark}[theorem]{Remark}

\theoremstyle{ass}

\newcommand{\Acal}{\mathcal{A}}

\newcommand{\Bcal}{\mathcal{B}}

\newcommand{\Dcal}{\mathcal{D}}

\newcommand{\Hcal}{\mathcal{H}}

\newcommand{\Scal}{\mathcal{S}}

\newcommand{\Zcal}{\mathcal{Z}}

\newcommand{\Ibb}{\mathbb{I}}

\newcommand{\Rbb}{\mathbb{R}}

\DeclareMathOperator{\diverg}{div}

\DeclareMathOperator{\dist}{dist}
\DeclareMathOperator{\rank}{rank}

\DeclareMathOperator{\supp}{supp}

\newcommand{\norm}[1]{\|#1\|}

\newcommand{\normb}[1]{\bigl\|#1\bigr\|}

\newcommand{\abs}[1]{|#1|}

\newcommand{\absB}[1]{\Bigl|#1\Bigr|}

\newcommand{\dd}{\;\mathrm{d}}

\newcommand{\N}{\mathbb{N}}
\newcommand{\R}{\mathbb{R}}

\newcommand{\weakly}{\rightharpoonup}
\newcommand{\weaklystar}{\overset{*}\rightharpoonup}

\newcommand{\eps}{\epsilon}
\newcommand{\length}{l}

\def\XXint#1#2#3{{\setbox0=\hbox{$#1{#2#3}{\int}$} 
\vcenter{\hbox{$#2#3$}}\kern-.5\wd0}}

\title[Asymptotic spectral analysis in semiconductor nanowire heterostructures]{Asymptotic spectral analysis in \\semiconductor nanowire heterostructures}

\author{Carolin Kreisbeck}
\address{Fakult\"at f\"ur Mathematik, 
Universit\"at Regensburg, 
93040 Regensburg, Germany}
\email{carolin.kreisbeck@mathematik.uni-r.de}

\author{Lu\'{i}sa Mascarenhas}
\address{Departamento de Matem\'{a}tica and Centro de Matem\'{a}tica e Aplica\c{c}\~{o}es, 
Faculdade de Ci\^{e}ncias e Tecnologia, Universidade Nova de Lisboa, 
Quinta da Torre, 2829-516 Caparica, Portugal}
\email{mascar@fct.unl.pt}

\begin{document}
%
%
%
%
\begin{abstract}
Mathematical settings in which heterogeneous structures affect electron\break transport 
through a tube-shaped quantum waveguide are studied, highlighting the interaction between heterogeneities and geometric parameters 
like curvature and torsion.
First, the macroscopic behavior of a nanowire made of composite fibers with microscopic periodic texture is analyzed, which
amounts to determining the asymptotic behavior of the spectrum of an
elliptic Dirichlet eigenvalue problem with finely oscillating coefficients
in a tube with shrinking cross section. A suitable formal expansion suggests that the effective one-dimensional limit problem is of 
Sturm-Liouville type and yields the explicit formula for the underlying potential. In the torsion-free case, these findings are made rigorous 
by performing homogenization and 3d-1d dimension reduction for the two-scale problem in a variational framework by means of $\Gamma$-convergence. 
Second, waveguides with non-oscillating inhomogeneities in the cross section are investigated. This leads to explicit criteria for 
propagation and localization of eigenmodes. 

\vspace{8pt}


\noindent\textsc{Keywords:}\hspace{0.02cm}  spectral analysis, dimension reduction, homogenization, $\Gamma$-convergence, quantum waveguides.

\vspace{8pt}

\noindent\textsc{Date:} \today.
\end{abstract}

\maketitle


\section{Introduction}
Among various applications~\cite{BullMaterSci07, Cao2004}, nanowires are considered promising building blocks for new 
developments in the construction of computing devices~\cite{Nature11} and 
solar energy conversion~\cite{MaterChem12}. Thus, the study of nanowires has stimulated an active field of research in the physics community over the recent years.      
With a diameter of the order of a nanometer, nanowires are almost one-dimensional objects which have special physical and chemical properties 
different from those of their bulk counterparts. It is well-known that these properties are sensitive to the material composition and to the geometry of the nanowire.
Both effects have been studied separately, experimentally~\cite{BullMaterSci07, Nature11, RoyalSoc, MaterChem12, Nature02, Nanoletter1} as well as in the 
mathematics literature~\cite{DucEx95, CDFK2005, BMT07, BoFr2010, Krejcirik12, AllPia05, Allaire08}.
Regarding the interplay between the two characteristics (material heterogeneities and geometry), however, there are hardly any quantitative results available.
Our goal is to gain new insight into the effect of heterogeneous structures on electron transport through a tube-shaped quantum waveguide, 
highlighting their interaction with curvature and torsion.
In the following, we address two set-ups both of which feature heterostructures in lateral direction, but with different intrinsic length-scales compared to
the thickness of the wire. 
First, we deal with nanowires with a microscopic periodic texture of composite fibers, and second, we study waveguides 
with a cross section made of various material components, such as the experimentally 
relevant core-multi-shell nanowire heterostructures~\cite{Nature02, RoyalSoc, BullMaterSci07, Nature11}. 
In both cases, effective one-dimensional limit problems are derived explicitly by the use of analytical tools. Our rigorous analysis can serve as a basis for 
numerical simulations of sophisticated devices involving heterogeneous bent and twisted nanowires.

Ballistic transport in modulated semiconductor devices, such as nanowire heterostructures, is governed by the effective-mass Schr\"odinger equation
\begin{align}\label{effmass_Schroedinger}
-\frac{\hslash^2}{2}\diverg \Bigl(\frac{1}{m(z)}{\nabla \psi(z)}\Bigr) + V(z)\,\psi(z)=E\psi(z), \qquad z\in \R^3,
\end{align}
where $V$ is a sharp potential that is zero inside the confinement imposed by the device geometry and infinite outside. The quantities 
$\psi$ and $E$ stand for the wave function and the energy, respectively.
The spatial position-dependence of the effective mass $m$ allows to model material consisting of different components. 
In what follows, we use a notation that is more common in the mathematical literature and rewrite~\eqref{effmass_Schroedinger} as an 
elliptic Dirichlet eigenvalue problem of the form
\begin{align}\label{EVP}
\begin{cases}
     -\diverg (A \nabla u)=\lambda u & \ {\rm in\ } \Omega,\\
   \ \hspace{1.55cm}{u=0\ }& \ {\rm on\ } \partial\Omega,\\ 
 \end{cases}
 \end{align}
where $\Omega\subset\R^3$ is the domain occupied by the 
quantum waveguide, $A=A(z)$ with $z\in \Omega$ encodes the material heterogeneities, $u$ is the wave function, and $\lambda$ the corresponding energy level. 
In this work we consider $\Omega$ to be a curved and twisted thin tube modeling a nanowire, and we assume that $A$ is constant in the 
longitudinal direction. 
 
To specify the underlying length-scales of the problem,
we introduce the parameters $\delta>0$ and $\eps>0$.
Precisely, $\delta$ refers to the thickness of the nanowire and $\eps$ represents the length-scale of the periodic material heterogeneities. 
Let $\Omega_\delta$ be a tube of finite length $l$ with cross section $\delta\omega$, where $\omega$ is a fixed regular and bounded domain in $\R^2$. 
Let $A _{\eps,\delta}$ represent the material properties, which are supposed to vary only on the cross section $\delta\omega$. We will define $A _{\eps,\delta}$ later 
in Section~\ref{subsec:geometry} after the geometry of $\Omega_\delta$ has been specified. 
Then, the eigenvalue problem~\eqref{EVP} becomes
\begin{align}\label{EVPepsilondelta_nonrescaled}
\begin{cases}
     -\diverg (A _{\eps,\delta} \nabla u_{\eps, \delta})=\lambda_{\eps, \delta} u_{\eps, \delta} & \ {\rm in\ } \Omega_\delta,\\
   \hspace{2cm}{u_{\eps,\delta}=0\ }& \ {\rm on\ } \partial\Omega_\delta.\\ 
 \end{cases}
 \end{align}
Since $\Omega_{\delta}$ is bounded, the spectrum $\Sigma_{\eps,\delta}$ of \eqref{EVPepsilondelta_nonrescaled} is 
discrete and one may write $\Sigma_{\eps,\delta}=\{\lambda_{\eps,\delta}^{(j)}\}_{j\in \N_0}$, where 
$0<\lambda_{\eps,\delta}^{(0)} < \lambda_{\eps,\delta}^{(1)}\leq \lambda_{\eps,\delta}^{(2)}\leq \ldots$, with each eigenvalue repeated 
according to its multiplicity, and $\displaystyle \lim_{j\to \infty}\lambda_{\eps,\delta}^{(j)}=\infty$.

The goal is to derive the asymptotic behavior of the spectrum $\Sigma_{\eps, \delta}$ of \eqref{EVPepsilondelta_nonrescaled} as $\eps$ and $\delta$ tend to zero.
In fact, the result depends crucially on the relation between the two parameters. 
The regime $\delta=1$ corresponds to a classical homogenization problem in the fix
domain $\Omega_1$ as for instance in~\cite{JikovKozlovOleinik94,CioDon99},
while the case $\eps=1$ models thin tubes with inhomogeneous, but non-oscillating, cross section.  
An interaction between the effects of homogenization and dimension reduction is observed when the two parameters are of the same
order, i.e.~$\delta=\eps$.
A hint on the limiting behavior of the spectrum is obtained using the method of 
asymptotic expansion, which has proven very efficient in the periodic framework~\cite{BenLionsPapa78, Vanninathan81, CioDon99}. 

If the transversal heterogeneities oscillate at the same rate as the thickness of the wave\-guides, 
these formal computations suggest that the spectrum of \eqref{EVPepsilondelta_nonrescaled} is governed
by a contribution 
of order $\eps^{-2}$ resulting from the homogenization of the two-dimen\-sional eigenvalue problem in the cross section, and a zero-order term 
representing the effective propagation along the center curve of the waveguide.
Precisely, for an element $\lambda_\eps^{(j)}\in \Sigma_\eps:=\Sigma_{\eps, \eps}$ with $j\in \N_0$ one finds
\begin{align}\label{intro:expansion}
 \lambda_\eps^{(j)}=\frac{\mu_H}{\eps^2} + \eta_P^{(j)} + O(\eps).
\end{align}
Here, $\mu_H$ stands for the first eigenvalue of the homogenized cross section problem and $\{\eta_P^{(j)}\}_{j\in \N_0}$ represents the spectrum of a
one-dimensional Sturm-Liouville problem, the potential of which is given explicitly in terms of the waveguide geometry and the material properties.
In this work we provide the rigorous proof of~\eqref{intro:expansion} for untwisted waveguides, while the exact justification of the 
formal expansion remains open in the general case with torsion. 
Indeed, by formulating the two-scale problem~\eqref{EVPepsilondelta_nonrescaled}
in a variational framework and by applying $\Gamma$-convergence, we perform simultaneously homogenization and 3d-1d dimension reduction. 
The detailed statement of the result is 
formulated in Theorem~\ref{theo:mainresult_cylinder}, which partially generalizes the findings for homogeneous waveguides in~\cite{BMT07} 
(see also the references therein). 
Our result shows that, qualitatively, transport through a nanowire made of microscopically periodic fibers resembles 
transport through a homogeneous nanowire in the sense that both give rise to propagation of electrons.

For waveguides with non-oscillating inhomogeneities, however, the situation is different.
The interesting finding is that propagation behavior occurs if an explicit propagation criterion is satisfied, otherwise, electron transport
will be dominated by localization around specific points along the wire. A similar phenomenon was discovered for waveguides with
varying cross section~\cite{FriedlanderSolomyak07, FriedlanderSolomyak09} and homogeneous waveguides with Robin boundary conditions~\cite{BMT_Robin}.
Our propagation criterion requires a certain function $h$, defined on the central curve of the wire, to be constant. 
The function $h$ relates the geometric data of the wire with an expression 
measuring the symmetry of the cross section and its material structure. The precise condition, together with the spectral convergence 
of \eqref{EVPepsilondelta_nonrescaled} to a Sturm-Liouville type problem, is stated in Theorem~\ref{theo:inhomo_propagation}.
For point symmetric cross sections, $h$ vanishes identically, independently of the geometry. So, from the viewpoint of applications,
symmetrically constructed nanowires are propagation robust devices with respect to bending and twisting. For a non-symmetric material distribution in the cross section,
there are sharp conditions on curvature and on torsion that still allow electrons to propagate.  
If $h$ attains its global minimum at isolated points, so that in particular the propagation criterion is violated, one observes a concentration of the 
eigenmodes of \eqref{EVPepsilondelta_nonrescaled} around these points.
After an appropriate blow-up, we investigate the effective local behavior of the system to find that 
the eigenmodes of~\eqref{EVPepsilondelta_nonrescaled} behave like the eigenmodes of a one-dimensional harmonic oscillator.   
In conclusion, we characterize quantitatively the relation between heterogeneous material structures and device geometry for transport through nanowires.
In particular, we give explicit conditions saying if and how
electron transport can be controlled by simple geometrical operations. This can be viewed as a theoretical concept for the building of mechanical
quantum switching devices.

The paper is structured as follows. In Section~\ref{sec:preliminaries}, after fixing notation, we present some preliminary results that specify 
the setting and provide important technical tools. 
We start with the modeling of the waveguide geometry using the Tang frame in Section~\ref{subsec:geometry}, and introduce suitable energy functionals to reformulate our problem 
in a variational context. In Section~\ref{subsec:change}, as it is common in dimension reduction problems, a
change of variables is performed, which allows us to work on a fixed domain. 
We recall briefly the formal method of asymptotic expansion in Section~\ref{subsec:crosssection}, and explain the 
$\Gamma$-convergence approach to spectral asymptotic analysis in Section~\ref{subsec:Gamma}. 
Section~\ref{sec:case_oscillatingcoefficients} is concerned with the study of waveguides with microscopic periodic fibers, 
while Section~\ref{sec:case_inhomogeneous} 
addresses transport through nanowires 
with inhomogeneous cross sections. In both cases, the first essential step is to determine the asymptotics of the 
corresponding cross section eigenvalue problem
(see~Sections~\ref{subsec:asymptotics} and~\ref{subsec:asymptotics_inhomogeneous}, respectively). The statements of
the effective one-dimensional propagation models, along with their proofs, are presented in Sections~\ref{subsec:mainresult} and~\ref{subsec:propagation}.
We discuss various results regarding 
localization phenomena in Section~\ref{subsec:localization}.

\section{Preliminaries}\label{sec:preliminaries}

\subsection{Notations}\label{subsec:notations}
We use the standard notation for Lebesgue and Sobolev spaces.
Partial derivatives with respect to the $j$th component of a variable $x$ are denoted by $\partial_j^{x}$. 
If it is clear from the context, we will omit the variable, writing simply $\partial_j$. The same applies to other partial differential 
operators like $\diverg$, $\nabla$ or $\Delta$.

Regarding eigenvalues and eigenfunctions, the superscript notation $(\lambda^{(j)}, u^{(j)})$ with ${j\in \N_0}$ will be 
used to refer to the eigenpair of order $(j+1)$. 
Throughout this paper we address elliptic spectral problems on bounded domains.
So, $\{\lambda^{(j)}\}_j$ is always a non-decreasing sequence, and 
$\{u^{(j)}\}_j$ may be chosen to form an orthonormal basis of eigenfunctions for the $L^2$-space.
To simplify notation we often skip the index when dealing with the first eigenpair, 
meaning that we write $\lambda$ and $u$ instead of $\lambda^{(0)}$ and $u^{(0)}$.

In what follows, we apply the sum convention on repeated indices.
Note that subsequences are not relabeled, and that the actual value of a constant may 
change from line to line.
\vskip2mm

\subsection{Structure of the bent and twisted waveguide}\label{subsec:geometry}
Let $r: [0,\length]\to \R^3, s\mapsto r(s)$ be a simple $C^2$-curve parametrized by arc length. 
Denoting by $v'$ the derivative of $v$ with respect to the variable $s$, $T=r'$ will stand for the unitary tangent vector.
The commonly used Frenet local system of coordinates $(T, N, B)$, where $N$ and $B$ represent the normal and the binormal, 
respectively, does not only require the curve to be $C^3$, but also has two disadvantages concerning our approach: 
first, in order to define the normal vector $N$, one has to impose $T'\not= 0$, excluding straight lines segments, 
second, the two-dimensional system $(N,B)$ rotates around the tangent direction. Since we are concerned with bending 
(curvature) and twisting (torsion) effects, it is desirable to allow the curvature to vanish and, in order to keep better control 
of the twisting parameter, to choose a frame without intrinsic rotation around $T$. Therefore, we consider as reference system the 
Tang frame (see~\cite{Tang70, Bishop75}), i.e.~an orthonormal, positively oriented system $(T, X, Y)=(T(s), X(s), Y(s))$ satisfying, for all $s\in[0, \length]$,
\begin{align*}
 X'(s)=\zeta (s) T(s)\, ,\quad Y'(s)\,= \vartheta (s) T (s)\,, \quad T'(s) =-\zeta(s)  X(s) - \vartheta (s) Y(s).
\end{align*}
In fact, with the curvature denoted by $k$ ($k(s):=\abs{r''(s)}$), 
it turns out that such a system of differential equations can be solved for $r$ merely of class $C^2$ by taking
\begin{align*}  
\zeta (s)= -k(s)\cos\alpha (s)\quad  {\rm and}\quad  \vartheta (s) = - k(s)\sin\alpha(s)\,,\qquad s\in [0,\length],
\end{align*}
with a suitable function $\alpha$, which represents the angle of rotation of $(N, B)$ around $T$, whenever the Frenet system is well-defined. 
We also refer to the recent work ~\cite{ Krejcirik12}, where the existence of a Tang frame is established even for $r \in W^{2,\infty}(0, l; \R^3)$.

Let $\omega\subset \Rbb^ 2$ be a bounded and connected Lipschitz domain. After a convenient zoom, $\omega$ designates the 
cross section of the waveguides
under consideration. 
Let $\theta=\theta(s)$ be a given rotation angle that is assumed to be regular enough. For $\delta >0$, we may define the following domain:
\begin{align*}
 \Omega_\delta = \{z\in \R^3: z=r(s) +\delta x_1 X_\theta(s)+\delta x_2 Y_\theta(s),\ s\in [0,\length],\ x=(x_1, x_2)\in \delta\omega\},
\end{align*}
where $X_\theta(s)=\cos\theta(s) X(s)+\sin\theta(s)Y(s)$ and $Y_\theta(s)=-\sin\theta(s)X(s)+\cos\theta(s)Y(s)$ for $s\in[0,\length]$. 
In fact, if $\delta$ is sufficiently small, $\Omega_\delta$ is tube-shaped with central curve $r$ of length $\length$ and cross section $\delta\omega$ 
twisted around $r$ by an angle $\theta$.

Apart from the shape of the cross section $\omega$, two other geometric parameters will play an important role in our spectral analysis: 
the curvature $k$ and the torsion $\tau:=\theta'$.

Now that the geometry of $\Omega_\delta$ is settled, we define the matrix $A _{\eps,\delta}$ as follows. Let $Y:=(0, 1)^2$ and let $A$ represent 
a $Y$-periodic and measurable function 
defined in $\mathbb{R}^2$, taking values in the set of all symmetric $2\times 2$ matrices. We suppose that $A$ is bounded and uniformly coercive, i.e., 
there exist two positive constants $c$ and $d$ such that
\begin{align*}
c |\eta|^2\le A(y)\eta\cdot\eta\le d|\eta|^2
\end{align*}
for almost every $y\in Y$ and for all $\eta\in\mathbb{R}^2.$ Then, $A(x/\eps)$ is an $(\eps Y)$-periodic function of $x\in\mathbb{R}^2$. 
We define for all $z\in \Omega_\delta$, recalling that $z=r(s) +\delta x_1 X_\theta(s)+\delta x_2 Y_\theta(s)$,
\begin{align*}
A _{\eps,\delta} (z):= A(x/\eps), \quad x=( x_1, x_2)\in\omega.
\end{align*}
For the sake of simplicity, we will discuss only the isotropic case, meaning that we consider $A (y)= a (y)\Ibb$ with $a:\R^2\to \R$ $Y$-periodic and $\Ibb$ the 
identity matrix in $\R^{2\times 2}$. 
The general case can be treated with minor changes. In the sequel, we will use the notation $a_\eps$ for $a(\cdot/\eps)$.

\subsection{Change of variables and the rescaled energy}\label{subsec:change}
We consider the eigenvalue problem \eqref{EVPepsilondelta_nonrescaled} in $ \Omega_\delta$ and denote by $E_{\eps, \delta}$ the corresponding energy functional 
defined in $H^1_0 (\Omega_\delta)$ through
\begin{align*}
E_{\eps, \delta} [u]:=\int_{\Omega_\delta} A_{\eps,\delta} \nabla u\cdot \nabla u\, \dd{z}.
\end{align*}

For our analysis of the asymptotic behavior of~\eqref{EVPepsilondelta_nonrescaled} as $\delta$ and $\eps$ go to zero, 
we first perform a change of variables, transforming the problem defined in $\Omega_\delta$ into a problem on a fixed domain. 
Then, using $\Gamma$-convergence, we investigate the limit behavior of the 
resulting energy functionals. 

Consider the following change of variables
\begin{align}\label{rescaling} 
\psi_\delta: Q_\length:=(0, \length)\times\omega &\to\Omega_\delta \nonumber\\
 (s, (x_1, x_2))&\mapsto z=r(s) + \delta x_1 X_\theta(s)+\delta x_2 Y_\theta(s),
\end{align}
which transforms the straight cylinder of length $\length$ with cross section $\omega$ into the thin curved and 
twisted domain $\Omega_\delta$. Accounting for the definitions of the previous section we obtain, in the Tang frame $(T, X, Y)$, that
\begin{align*}
 \nabla\psi_\delta (s,x)=\left(\begin{array}{ccc} \beta_\delta & 0 & 0\\\delta \tau(R z_\theta \cdot x)& \delta \cos\theta & -\delta\sin\theta\\ 
\delta \tau(z_\theta\cdot x)& \delta \sin\theta & \delta \cos \theta \end{array}\right) \qquad \text{with $\det \nabla \psi_\delta=\delta^2\beta_\delta$,}
\end{align*}
where $z_t:=(\cos t, -\sin t)$ for $t\in\R$, $R$ is the rotation in 
$\R^2$ by the angle $-\pi/2$, $\tau=\theta'$, and
\begin{align}\label{beta} 
\beta_\delta(s,x) = 1 - \delta \big(\xi(s)\cdot x\big)\qquad \text{with $\xi(s)=k(s) z_{[(\theta-\alpha)(s)]}$}.\end{align}

Throughout this work
we assume $\xi:[0,\length]\to \R^2$ to be Lipschitz continuous. 
For the inverse one finds
 \begin{align*}
 (\nabla\psi_\delta)^{-1}(s,x)=\left(\begin{array}{ccc} \beta_\delta^{-1} & 0 & 0\\ \beta_\delta^{-1}\tau x_2 & \delta^{-1}\cos\theta  
& \delta^{-1} \sin\theta\\ 
-{\beta_\delta^{-1}}{\tau x_1}& - \delta^{-1}\sin\theta & \delta^{-1}\cos\theta \end{array}\right).
\end{align*}
After dividing by $\delta^2$ and having in mind that $ A_{\eps,\delta}(z)= a (x/\eps)\Ibb=a_\eps (x)\Ibb$, the energy functional $E_{\eps, \delta}[u]$
for $u\in H^1_0(\Omega_\delta)$ becomes
 \begin{align}\label{changedenergy}
\tilde{E}_{\eps, \delta} [v] = 
\int_{Q_\length}\left( \frac{a_\eps}{\beta_\delta} \abs{v' + \tau (\nabla^x v \cdot Rx)}^2 + 
\frac{a_\eps\,  \beta_\delta }{\delta^{2}} \abs{\nabla^x v}^2\right) \dd{s}\dd{x},
\end{align}
where
$$v(s,(x_1, x_2))=u(\psi_\delta(s,(x_1, x_2))\in H^1_0(Q_\length),$$  
and the gradient of $v$ is written in the form $(v', \nabla^x v)$. 

After the change of variables \eqref{rescaling}, problem~\eqref{EVPepsilondelta_nonrescaled} reads
\begin{align}\label{rescaledEVP}
\begin{cases}
\Acal_{\eps, \delta}\, v_{\eps, \delta}= \lambda_{\eps, \delta} \beta_\delta v_{\eps, \delta}& \ {\rm in\ } Q_\length,\\
\hspace{0.7cm}{v_{\eps, \delta}=0\ }& \ {\rm on\ } \partial Q_\length,\\ 
 \end{cases}
\end{align}
where \begin{subequations}
\begin{align}\label{EVPepsilondelta}
\Acal_{\eps, \delta}\, = -\diverg\Bigl(\frac{a_\eps}{\beta_\delta}(d\otimes d) \nabla \cdot\Bigr)- \frac{1}{\delta^2} 
\diverg^x\bigl(\beta_\delta\, a_\eps \nabla^x \cdot\bigr)
\end{align}
 with $d(s,x)= (1, \tau Rx)= (1, \tau x_2, -\tau x_1)\in \R^3$.
We remark that in the case without torsion, i.e.~$\tau=0$, the operator $\Acal_{\eps,\delta}$ reduces to
\begin{align}\label{EVPepsilondelta_notorsion}
 \Acal_{\eps, \delta}= -\diverg\Bigl(\frac{a_\eps}{\beta_\delta} (\cdot)'\Bigr)- \frac{1}{\delta^2} 
\diverg^x\bigl(\beta_\delta\, a_\eps \nabla^x \cdot\bigr).
\end{align}
\end{subequations}

As mentioned in the Introduction, we address two different cases of non-homogeneous cross section: first, in Section~\ref{sec:case_oscillatingcoefficients}, 
the case where the properties of the 
cross section oscillate at the same rate as the thickness of the domain, which means that we may set $\delta = \eps$. Here, we solve 
both an homogenization and a dimension reduction problem with the same parameter. Second, in Section~\ref{sec:case_inhomogeneous}, 
the case where the properties of the cross section vary, 
but do not oscillate, which corresponds to having $\eps$ identically equal to one, while $\delta$ decreases to zero. The limit behavior for the 
two cases turns out to be completely different.

\subsection{The cross section problem and the asymptotic method}\label{subsec:crosssection}
The expression~\eqref{changedenergy} of the rescaled energy shows the importance of understanding the order of 
convergence of the first eigenvalue of the following spectral problem in the cross section:
\begin{align*}
\begin{cases}
-\diverg \bigl( a_\eps \beta_\delta \nabla w \bigr) = \mu \beta_\delta w & \ {\rm in\ } \omega,\\
\hspace{2.1cm}{w=0\ }& \ {\rm on\ } \partial \omega.\\
\end{cases}
\end{align*}

A powerful tool in that direction is the method of asymptotic expansion.
This method consists in assuming that the solution of the parametrized problem has a certain development in powers of the involved parameters. 
Plugging this ansatz into the equation and gathering the terms of the same power leads to a series of problems, the solutions of which determine the 
coefficients of the development. This formal procedure needs to be justified by appropriate estimates. 
The following lemma is of great importance to obtain such estimates. For the proof we refer for instance to~\cite[Lemma 1.1]{JikovKozlovOleinik94}.
\begin{lemma}\label{lem:VishikLusternik}
Let $H$ be a Hilbert space equipped with the norm $\norm{\cdot}_H$ and let $L:H\to H$ be a 
continuous, linear, compact and self-adjoint operator. 
Suppose $\lambda>0$ and $u\in H$ are such that $\norm{Lu-\lambda u}_H\leq c \norm{u}_H$ for a constant $c>0$. \\
Then there is an eigenvalue
$\lambda^{(j)}$, $j\in\N_0$, of $L$ with $\abs{\lambda-\lambda^{(j)}}<c$. Besides, for any $d>c$ there exists $\bar{u}\in H$ with 
$\norm{\bar{u}}_H=\norm{u}_H$ 
that is the linear combination of eigenvectors associated with eigenvalues of $L$ lying in $[\lambda-d, \lambda+d]$ and that 
satisfies $\norm{u-\bar{u}}_H\leq 2\,(c/d)\norm{u}_H$.
\end{lemma}

\subsection{The $\Gamma$-convergence approach}\label{subsec:Gamma} 
In parametrized eigenvalue problems it is often hard to derive convergence for eigenvalues other than the first. 
When dealing with the spectrum of self-adjoint operators with compact inverse, $\Gamma$-convergence 
of energy functionals turns out to be a powerful tool, as we can see from Lemma \ref{theo:Gammaconvergence}, stated below. 
In this paragraph we begin by recalling the definition of $\Gamma$-convergence and the one-to-one correspondence between 
non-negative lower semicontinuous quadratic forms and positive self-adjoint operators. 

Let $H$ be a Hilbert space and $F_\eps, F: H\to [0, + \infty]$. We say that the sequence $\{F_\eps\}_\eps$ $\Gamma$-converges to $F$ in $H$ as $\eps\to 0$ if the
following two conditions hold: 
\begin{itemize}
 \item[$(i)$] \ {\it (Lower bound)}\ For any $v$ and $\{v_\eps\}_\eps$  such that
$v_\eps\to v$  in $H$,
$\displaystyle \liminf_{\eps\to 0} F_\eps [v_\eps] \ge F[v];$
 \item[$(ii)$] {\it (Upper bound)}\ For every $v$, there exists a sequence
$\{\tilde v_\eps\}_\eps$ such that
$\tilde v_\eps\to v$ in $H$ and
$\displaystyle \limsup_{\eps\to 0} F_\eps[\tilde v_\eps] \le F[v].$
\end{itemize}
Up to a subsequence, such a  $\Gamma$-limit $F$ always exists and is lower semicontinuous. For further features of 
$\Gamma$-convergence theory, we refer to \cite{DalMaso_Gammaconvergence}.

Let $F: H\to [0, + \infty]$ be a non-negative lower semicontinuous quadratic form with domain $\Dcal(F)$ and $B$ its associated
bilinear form, i.e.\ the unique symmetric bilinear form $B: \Dcal(F)\times \Dcal(F)\to \R$ such that $F[u]=B(u,u)$ for all $u\in \Dcal(F)$. 
To each $F$ we associate one and only one positive self-adjoint 
operator $\Bcal: \overline{\Dcal(F)}\to H$ such that $F[u]=(\Bcal u,u)$ for all $u\in \Dcal(\Bcal)$, where 
\begin{align*}\Dcal(\Bcal)=\{u\in \Dcal(F): \exists f\in \overline{\Dcal(F)} : B(u,v)=(f,v)\ \forall v\in \Dcal(F)\}
\end{align*} and $\Bcal u=f$
(see \cite[Theorem~12.13]{DalMaso_Gammaconvergence}).
Moreover, the $\Gamma$-limit of a sequence of non-negative quadratic form is again a non-negative quadratic form and, in addition, 
lower semicontinuous (see \cite[Theorem~11.10]{DalMaso_Gammaconvergence}).

For the proof of the following lemma we refer to \cite[Theorem~3.1]{BMT07}. 
\begin{lemma}\label{theo:Gammaconvergence}
Let $H$ be a Hilbert space. For each $\eps>0$ let $H_\eps$ represent $H$ endowed with an inner product $(\cdot,\cdot)_{\eps}$ satisfying
$\ c_\eps\norm{v}_H^2 \leq (v,v)_\eps\leq d_\eps \norm{v}_{H}^2$ for all $v\in H$ and two real sequences such that 
$\displaystyle \lim_{\eps\to 0}c_\eps= \lim_{\eps\to 0} d_\eps=1$.\\
Suppose that $\Bcal_\eps: H_\eps\to H_\eps$ is a densely defined
self-adjoint operator for every $\eps>0$ and that $F_\eps:H_\eps\cong H\to \overline{\R}:=\R\cup\{+\infty\}$ is the corresponding (extended-valued) 
lower semicontinuous quadratic form such that $F_\eps[v]=(\Bcal_\eps v,v)$ for all $v\in \Dcal(\Bcal_\eps)$.

Further, assume that the following three conditions hold:
\begin{itemize}
 \item[$(i)$] (Lower bound) There is a constant $c>0$ (independent of $\eps$) such that $F_\eps[v]\geq -c\norm{v}_H^2$ 
for all $v\in H$.
 \item[$(ii)$] (Compactness) If $\{v_\eps\}_\eps$ is a bounded sequence in $H$ with \begin{align*} \sup_{\eps>0} F_\eps[v_\eps]\leq C<\infty,\end{align*} 
then there is a subsequence of $\{v_\eps\}_\eps$ converging in $H$.
 \item[$(iii)$] ($\,\Gamma$-convergence) It holds that $\Gamma(H)$-$\displaystyle\lim_{\eps \to 0} F_\eps=F_0$.
\end{itemize}
Then the limit functional $F_0:H\to \overline{\R}$ determines a unique closed
(not necessarily densely defined) operator $\Bcal_0: H\to H$ with compact resolvent such that $F_0[v]=\bigl(\Bcal_0 v,v\bigr)_H$ 
for all $v\in \Dcal(\Bcal_0)$.

Besides, there is convergence of the spectral problems associated with $\Bcal_\eps$ to the one associated with $\Bcal_0$. 
Precisely, this means:
Let $\{(\lambda_\eps^{(j)}, v_\eps^{(j)})\}_j$ and $\{(\lambda^{(j)}, v^{(j)})\}_j$ be the sequences of  
eigenpairs for
\begin{align*}
 \Bcal_\eps v_\eps= \lambda_\eps v_\eps,\quad v_\eps\in H_\eps, \qquad \text{and}\qquad \Bcal_0 v= \lambda v, \quad v\in H,
\end{align*}
respectively.
Then, for every $j\in \N_0$ it holds that
\begin{align*}
\lim_{\eps\to 0} \lambda_\eps^{(j)}= \lambda^{(j)}.
\end{align*}
After extracting a subsequence, $\{v_\eps^{(j)}\}_\eps$ converges strongly to eigenvectors 
corresponding to $\lambda^{(j)}$ as $\eps\to 0$. Conversely, any eigenvector 
$v^{(j)}$ of $\Bcal_0$ can be approximated in $H$ by a sequence of eigenvectors of $\Bcal_\eps$ to the eigenvalues $\lambda_\eps^{(j)}$. 
\end{lemma}

\section{Waveguides with microscopically periodic fibers} \label{sec:case_oscillatingcoefficients}
In this section we treat the case where the heterogeneities in the cross section oscillate at the same rate 
as the thickness of the domain shrinks, i.e.,\ we set $\delta = \eps$, so that the energy functional~\eqref{changedenergy} reads for $v\in H_0^1(Q_l)$,
\begin{align}\label{epschangedenergy}
\tilde{E}_{\eps, \eps}[v] = 
\int_{Q_l}\left( \frac{a_\eps}{\beta_\eps} \abs{v' + \tau (\nabla^x v \cdot Rx)}^2 + 
\frac{a_\eps\, \beta_\eps} {\eps^{2}} \abs{\nabla^x v}^2\right) \dd{s}\dd{x}.
\end{align}
Recall that $Q_l=(0, l)\times\omega$ and $a_\eps=a(\cdot/\eps)$ with $a$ being measurable in $\R^2$ and $Y$-periodic with $Y=(0, 1)^2$. Additionally, we assume that
$a$ satisfies 
\begin{align}\label{uniformbounds}
0<c\leq a(y) \leq d
\end{align}
for almost every $y\in Y$ with constants $c,d>0$.

\subsection{Asymptotic analysis of the cross section eigenvalue problem}\label{subsec:asymptotics} 
In order to find the asymptotic behavior of the rescaled energy, expression \eqref{epschangedenergy} 
suggests to investigate the precise rate of convergence of the eigenvalues for the cross section problem
\begin{align}\label{epsperturbed_crossproblem}
-\diverg \bigl(a_\eps \beta_\eps(s) \nabla w \bigr) = \mu \beta_\eps(s) w, \qquad w \in H_0^1(\omega),
\end{align} 
with $\beta_\eps(s)=\beta_\eps(s, \cdot),\ s\in [0, l], $ given by \eqref{beta}. Towards this end, we perform the formal asymptotic expansion described below.



\subsubsection{Formal expansion in the cross section}\label{subsubsec:formal_expansion}
Let $s\in [0, \length]$ be fixed, and consider for $\mu_\eps^\ast$ and $w_\eps^\ast$ an ansatz of the form
\begin{align*}
&\mu_\eps^\ast (s) = \mu_0(s) + \eps\mu_1(s) + \eps^2 \mu_2(s) , \\
&w_\eps^\ast(s,x) = w_0(s, x, y)+\eps w_1(s,x,y)+\eps^2 w_2(s, x,y) +\eps^3 w_3(s,x,y) + \eps^4 w^4(s,x,y),\nonumber
\end{align*}
where $y=x/\eps$ represents the fast scale with respect to $x\in\omega$. 
After plugging $\mu_\eps^\ast$ and $w_\eps^\ast$ into~\eqref{epsperturbed_crossproblem} 
and gathering power-like terms of order $\eps^{-2}$, $\eps^{-1}$, $\eps^0$, $\eps^1$ and $\eps^2$, we obtain the following equations, 
which will determine the expressions $\mu_i$ and $w_i$ in the developments of $\mu_\eps^\ast$ and $w_\eps^\ast$, respectively,
\begin{align}\label{eq:diffequations_order}
&(\eps^{-2})\ \  -\diverg^y(a\,\nabla^y w_0)=0; \nonumber\\
&(\eps^{-1})\ \  -\diverg^y(a\,\nabla^y w_1) - \diverg^y(a\nabla^x w_0)=0;\nonumber\\
&(\eps^{0})\ \ \ \ -\diverg^y(a\,\nabla^y w_2) - \diverg^y(a\nabla^x w_1) - \diverg^x\bigl(a(\nabla^y w_1 +\nabla^x w_0)\bigr) - \mu_0 w_0 =0;\\
&(\eps^{1})\ \ \ \ -\diverg^y(a\,\nabla^y w_3) - \diverg^y(a\nabla^x w_2) - \diverg^x\bigl(a(\nabla^y w_2 +\nabla^x w_1)\bigr) \nonumber\\&
\qquad\qquad\qquad\qquad\qquad\qquad\qquad\qquad\qquad+  a(\nabla^x w_0 + \nabla^y w_1)\cdot \xi- \mu_0 w_1 - \mu_1 w_0=0;\nonumber\\
&(\eps^{2})\ \ \ \ -\diverg^y(a\,\nabla^y w_4) - \diverg^y(a\nabla^x w_3) - \diverg^x\bigl(a(\nabla^y w_3+\nabla^x w_2)\bigr) \nonumber\\ 
&\qquad\quad + a(\nabla^x w_1 + \nabla^y w_2)\cdot \xi + a(\xi\cdot x)(\nabla^x w_0 + \nabla^y w_1)\cdot \xi  - \mu_0 w_2 - \mu_1 w_1 - \mu_2 w_0=0.\nonumber
\end{align}

Before solving the sequence of problems in \eqref{eq:diffequations_order}, some details are needed.
In fact, from classical homogenization theory (see \cite{Vanninathan81, KesaI79, KesaII79}) we expect that, for fixed $s$, the eigenvalues  
$\mu_\eps(s)$ of \eqref{epsperturbed_crossproblem} converge for $\eps\to 0$ to the 
eigenvalues of the homogenized problem
\begin{align}\label{homogenized_crossproblem}
-\diverg \bigl(Q  \nabla^x w \bigr) = \mu w, \qquad w \in H_0^1(\omega).
\end{align}
Here, the symmetric matrix $Q=(Q_{ij}),\, i,j\in \{1,2\}$ is given by the usual homogenization formulas. Precisely,
\begin{align}\label{definitionQ}
Q_{ij}= \bar{a} \,\delta_{ij} + \int_{Y}a\,\partial_j^y \phi_i\dd{y},
\end{align}
where $\bar{a}:=\int_Y a(y)\dd{y}$ and $\phi\in H_\#^1(Y;\R^2)$ with $\int_Y \phi\dd{y}=0$ solves the cell problem
\begin{align}\label{equation_psi}
 -\diverg^y(a \,\nabla^y \phi_i)= \partial_i^y a.
\end{align}
We denote the first eigenpair of \eqref{homogenized_crossproblem}  by $(\mu_H, w_H)$.
Krein-Rutman's theorem (see for instance~\cite[Chapter~VIII,~\S 4,~Appendix]{DautrayLions90}) yields that $\mu_H$ is real, positive and simple and that 
the corresponding first eigenfunction $w_H$ can be chosen to be strictly positive
with $\int_\omega w_H^2 \,{dx}=1$. Regularity results guarantee that $w_H$ lies in $H_0^1(\omega)\cap C^\infty(\bar{\omega})$.

We define the third-order tensor $P=(P_{ijk})$ with $i,j,k\in\{1,2\}$ by
\begin{align*}
 P_{ijk} = \int_Y a \delta_{ij}\phi_k+ a\partial_k^y \zeta_{ij}\dd{y},
\end{align*}
where $\zeta=(\zeta_{ij})\in H^1_\#(Y;\R^{2\times 2})$ is such that $\int_Y \zeta\, dy=0$ and solves
\begin{align*}
 -\diverg^y(a\nabla^y \zeta_{ij})=a \delta_{ij} + \partial_j^y(a\phi_i) + a \partial_j^y \phi_i-Q_{ij}.
\end{align*}
Further, we set $S=(S_{ijk})$ for $i,j,k\in \{1,2\}$ to be
\begin{align*}
 S_{ijk}=\int_Y a\partial_j^y \varkappa_{ik}(y) - a\partial_k^y
\zeta_{ij}(y) - a \phi_i\delta_{jk}\dd{y},
\end{align*}
with $\varkappa=(\varkappa_{ij})\in H_\#^1(Y;\R^{2\times 2})$ such that $\int_Y \varkappa\dd{y}=0$ and 
\begin{align*}
  -\diverg^y(a \,\nabla^y \varkappa_{ij}) =-a\delta_{ij}-a\partial_j^y \phi_i + Q_{ij},
\end{align*}
and $R=(R_{ijkl})$ for $i,j,k,l\in \{1,2\}$ as the fourth-order tensor
\begin{align*}
  R_{ijkl}= \int_{Y} a\delta_{ij} \zeta_{kl}(y) + a \partial_l^y \Lambda_{ijk}(y)\dd{y}.
 \end{align*}
Here, $\Lambda=(\Lambda_{ijk})\in H_\#^1(Y;\R^{8})$ satisfies $\int_Y \Lambda\dd{y}=0$
and is the solution of
\begin{align*}
 -\diverg^y(a \,\nabla^y \Lambda_{ijk})
=  a\delta_{ij}\phi_k + \partial _k^y(a\zeta_{ij}) + a\partial_k^y \zeta_{ij} - Q_{ij}\phi_k-P_{ijk}.
\end{align*}

Then, solving successively the differential equations in~\eqref{eq:diffequations_order}, having regards to the Fredholm compatibility conditions, we obtain that
\begin{align}\label{defw123}
 w_0(s,x,y) &= w_H(x), \nonumber\\
 w_1(s,x,y) &= \phi_i(y)\ \partial_i^x w_H(x) + \bar{w}_1(s,x), \\
 w_2(s,x,y) &= \zeta_{ij}(y)\ \partial_{ij}^x w_H(x) + \phi_i(y)\ \partial_i^x \bar{w}_1(s,x) + \bar{w}_2(s,x), \nonumber \\
 w_3(s,x,y) &=\! \Lambda_{ijk}(y) \partial_{ijk}^x w_H(x) + \zeta_{ij}(y) \partial_{ij}^x \bar{w}_1(s,x)
             + \phi_i(y)\partial_i^x \bar{w}_2(s,x)+ \varkappa_{ij}(y)\xi_j(s)\partial_i^x w_H(x) \nonumber
\end{align}
for fixed $s\in [0, l]$ and $x\in \omega$, and $w_4$ solves
\begin{align}\label{defw4}
-\diverg^y(a\,\nabla^y w_4) =& \diverg^y(a\nabla^x w_3) + \diverg^x (a\nabla^y w_3)+ \diverg^x(a\nabla^x w_2)  \\
                                              &- (a\nabla^x w_1) \cdot \xi - (a\nabla^y w_2)\cdot \xi - a(\nabla^x w_0)\cdot \xi\, (\xi\cdot x) - (a\nabla^y w_1)\cdot \xi\, (\xi\cdot x)  \nonumber\\
                                              &+ \mu_0 w_2 + \mu_1 w_1 + \mu_2 w_0. \nonumber     
\end{align}
Since it does not play a role for our limit parameters, we leave the details about
the explicit form of $w_4$ to Appendix~\ref{appendix_A1}.
In the above expressions $\bar{w}_1(s,x)= \bar{w}(x)+\xi(s)\cdot\hat{w}(x)$, where $\bar{w}$ is the solution of
\begin{align}\label{def:wbar1}
 -Q_{ij}\partial_{ij}^x \bar{w} -\mu_H \bar{w} = P_{ijk} \partial_{ijk}^x w_H, \qquad \bar{w}\in H_0^1(\omega), \quad i,j,k\in\{1,2\},
\end{align}
with $\int_\omega \bar{w} w_H\dd{x}=0$, while $\hat{w}$ solves
\begin{align}\label{def:what}
  -Q_{ij}\partial_{ij}^x \hat{w}_k -\mu_H \hat{w}_k = -Q_{ik} \partial_{i}^x w_H, \qquad \hat{w}\in H_0^1(\omega; \R^2), \quad i,j,k\in\{1,2\},
\end{align}
and satisfies $\int_\omega \hat{w} w_H\dd{x}=0$;
moreover, for $s$ fixed, $\bar{w}_2(s)=\bar{w}_2(s,\cdot)\in H_0^1(\omega)$ with 
$\int_\omega \bar{w}_2(s)w_H\dd{x}=0$
is defined as the solution of
\begin{align}\label{defbarw2}
 &-Q_{ij}\partial_{ij}^x\bar{w}_2(s) - \mu_H\bar{w}_2(s)= R_{ijkl}\partial_{ijkl}^x w_H + P_{ijk}\partial_{ijk}^x \bar{w}_1(s) + S_{ijk}\xi_k(s)\partial_{ij}^x w_H\\
&\qquad\qquad\qquad\qquad\qquad\qquad\qquad - Q_{ij}\xi_j(s) \partial_i^x \bar{w}_1(s)  
- Q_{ij}\xi_j(s) (\xi(s)\cdot x)\partial_i^x w_H + \mu_2(s) w_H.\nonumber
\end{align}

Concerning the coefficients of $\mu_\eps^\ast$ one has
\begin{align*}
 \mu_0(s)=\mu_H,\qquad
 \mu_1(s)=0,\qquad
 \mu_2(s)=q_H + q_{\xi}(s),
\end{align*}
where
\begin{align}\label{def:q_H}
 q_H = - R_{ijkl}\left( \int_\omega (\partial_{ijkl}^x w_H) w_H\dd{x}\right) + P_{ijk} \left(\int_{\omega} (\partial_{ijk}^x w_H)\bar{w}\dd{x}\right),
\end{align}
independent of $s$, represents the homogenization contribution at the second order, and
\begin{align}\label{def:q_xi}
q_\xi(s) &= P_{ijk} \left( \int_\omega (\partial_{ijk}^xw_H) \bigl(\hat{w}\cdot \xi(s)\bigr)\dd{x}\right) 
+ Q_{ij} \xi_j(s)\left( \int_\omega (\xi(s)\cdot x)(\partial_{i}^x w_H) w_H\dd{x}\right) \nonumber\\ 
&\qquad\qquad - Q_{ij}\xi_j(s) \left( \int_\omega (\partial_{i}^x w_H) \bar{w}\dd{x}\right)
 -Q_{ij}\xi_j(s) \left( \int_\omega (\partial_{i}^x w_H) \bigl(\hat{w}\cdot \xi(s)\bigr)\dd{x}\right) \\ 
&\qquad\qquad - S_{ijk}\xi_k(s)
\left( \int_\omega (\partial_{ij}^x w_H) w_H\dd{x}\right)\nonumber\\
&= -\frac{1}{4} Q_{ij} \xi_i(s)\xi_j(s)-2 Q_{ij}\xi_j(s) \left( \int_\omega (\partial_{i}^x w_H) \bar{w}\dd{x}\right) - S_{ijk}\xi_k(s)
\left(\int_{\omega} (\partial_{ij}^x w_H) w_H\dd{x}\right)\nonumber
\end{align}
reflects essentially the effects introduced  by the curvature. The last equality in \eqref{def:q_xi} is a consequence of the following lemma.
\begin{lemma}
Let $s\in [0,l]$ be fixed. It holds that $\displaystyle\int_\omega (\xi(s) \cdot x)(\partial_i^x w_H) w_H\dd{x}=-\frac{\xi_i(s)}{2}$ for $i\in \{1,2\}$. 
Moreover, for $i,j,k\in \{1,2\}$,
\begin{align*}
 Q_{ij}\xi_j(s) \int_{\omega}(\partial_i^x w_H)(\hat{w}\cdot \xi(s))\dd{x}= -\frac{1}{4}Q_{ij}\xi_i(s)\xi_j(s),
\end{align*}
and 
\begin{align}\label{PQrelation}
 P_{ijk} \left( \int_\omega (\partial_{ijk}^x w_H) \bigl(\hat{w}\cdot \xi(s)\bigr)\dd{x}\right)
=- Q_{ij}\xi_j(s) \left( \int_\omega (\partial_{i}^x w_H) \bar{w}\dd{x}\right).
\end{align}
\end{lemma}
\begin{proof}
The first part of the statement follows from integration by parts. For the second equality we may argue similarly to~\cite[Lemma~4.3]{BMT07}. 
Recalling the definition of $\hat{w}$ and $w_H$ and using the symmetry $Q_{ij}=Q_{ji}$ for $i,j\in\{1,2\}$, we get
\begin{align*}
& Q_{ij}\partial_j^x\bigl(w_H(\partial_i^x \hat{w}_k) - (\partial_i^x w_H) \hat{w}_k\bigr) \\
&\qquad\qquad\qquad=  Q_{ij}\bigl(w_H(\partial_{ij}^x \hat{w}_k) - (\partial_{ij}^x w_H)\hat{w}_k\bigr)
 + Q_{ij} \bigl((\partial_j^x w_H)(\partial_i^x \hat{w}_k) - (\partial_i^x w_H)(\partial_j^x \hat{w}_k)\bigr) \\
&\qquad\qquad\qquad= Q_{ik} (\partial_i^x w_H)w_H = \frac{1}{2} Q_{ik}\partial_i^x (w_H^2), \qquad k\in \{1,2\}.
\end{align*}
Consequently, for every $k\in \{1,2\}$,
\begin{align*}
 Q_{ij}\xi_j(s) \int_{\omega}(\partial_i^x w_H)(\hat{w}\cdot \xi(s))\dd{x}
&=Q_{ij}\xi_j(s) \xi_k(s)\int_{\omega}(\partial_i^x w_H)\hat{w}_k\dd{x}\\
&= \frac{1}{2} \xi_k(s)\int_\omega \partial_j(\xi(s)\cdot x)Q_{ij} \bigl((\partial_{i}^x w_H)\hat{w}_k- w_H(\partial_{i}^x \hat{w}_k) \bigr)\dd{x}\\
&= \frac{1}{2} \xi_k(s)\int_\omega (\xi(s)\cdot x) Q_{ij}\partial_j^x \bigl(w_H(\partial_{i}^x \hat{w}_k) - (\partial_{i}^x w_H)\hat{w}_k\bigr)\dd{x}\\
&= -\frac{1}{4}\xi_k(s)\int_\omega \partial_i(\xi(s) \cdot x) Q_{ik}w_H^2\dd{x}=-\frac{1}{4}Q_{ik}\xi_i(s)\xi_k(s)\\
&= -\frac{1}{4}Q_{ij}\xi_i(s)\xi_j(s).
\end{align*}
To obtain~\eqref{PQrelation} we exploit~\eqref{def:what} and~\eqref{def:wbar1}, which entails
\begin{align*}
P_{ijk} \xi_l(s) \int_\omega (\partial_{ij}^x w_H)(\partial_k^x \hat{w}_l )\dd{x} 
&= -\xi_l(s)\int_\omega \bigl(Q_{ij} \partial_{ij}^x \hat{w}_l + \mu_H \hat{w}_l\bigr)\bar{w}_1\dd{x}\\
&= -Q_{il} \xi_l(s) \int_{\omega}(\partial_i^x  w_H)\bar{w}_1\dd{x}.
\end{align*}
\end{proof}

Our approximating sequences $\mu^\ast_\eps$ and $w_\eps^\ast$ are now well-defined as follows:
\begin{align}\label{asymptotic_expansion} 
\mu^\ast_\eps(s) & = \mu_H + \eps^2 (q_H + q_\xi(s)),\\
w_\eps^\ast(s,x) & = w_H(x) + \eps w_1(s,x, y) + \eps^2 w_2(s,x,y) + \eps^3 w_3(s,x,y) + \eps^4 w_4(s,x,y), \qquad y=x/\eps.\nonumber
\end{align}

\subsubsection{Justification of the formal expansion}
The next result yields the exact rate of convergence for the first eigenvalue of the cross section problem \eqref{epsperturbed_crossproblem} as well as the limit behavior of the corresponding eigenmode. 
In the following, for $\eps>0$ and $s\in[0,\length]$, let $(\mu_\eps(s), w_\eps(s))$ be the first normalized eigenpair of \eqref{epsperturbed_crossproblem}. Here, 
normalized means that 
\begin{align}\label{normalization}
 \displaystyle \int_\omega \beta_\eps\abs{w_\eps}^2\dd{x}=1,
\end{align}
and we suppose $w_\eps(s)>0$ in $\omega$. 

To make the formal computations of the previous section rigorous, some more regularity of the involved functions is required. 
Precisely, we need the following two hypotheses to be satisfied:

\begin{itemize}
 \item[\textit{(H1)}] $\phi\in W_{\#}^{1, \infty}(Y;\R^2)$ (see \eqref{equation_psi});
 \item[\textit{(H2)}] $\xi\in W^{1, \infty}((0,l);\R^2)$ (see \eqref{beta}).
\end{itemize}

\begin{lemma}\label{lem:asymptotics_mueps}
If (H1) is satisfied and $\xi\in L^{\infty}((0, l);\R^2)$, there exists a constant $C>0$ such that for all $s\in[0,\length]$ and $\eps>0$ sufficiently small, 
\begin{align*}
 \abs{\mu_\eps(s)-\mu_\eps^\ast(s)}\leq C\eps^3,
\end{align*}
and 
\begin{align*}
 \norm{w_\eps(s)-w_\eps^\ast(s)}_{H^1(\omega)}\leq C\eps^{1/2}.
\end{align*}
\end{lemma}

\begin{proof}
Let $s\in [0,\length]$ be fixed. For the sake of simplicity we will drop the explicit dependence on $s$ in the subsequent computations. 
The basic idea for deriving the stated estimates is to apply Lemma~\ref{lem:VishikLusternik}. Before doing so in Step~3, we provide the necessary
preliminary work in Steps~1 and~2.

{\textit{Step~1: An upper bound estimate.}}
Since the terms in the development~\eqref{asymptotic_expansion} are defined so that, by plugging $\mu^\ast_\eps$ and $w^\ast_\eps$ 
into \eqref{epsperturbed_crossproblem}, all coefficients of order less than three cancel, one finds 
\begin{align}\label{eq:remainder}
 \diverg \bigl(a_\eps \beta_\eps  \nabla w^\ast_{\eps}\bigr) + \mu^\ast_\eps \beta_\eps w^\ast_\eps &=
\eps^3 \,\bigl[\diverg^x\bigl(a\,(\nabla^x w_3 + \nabla^y w_4 - (\xi\cdot x)(\nabla^x w_2 + \nabla^y w_3))
\bigr)\nonumber\\
&\qquad\qquad+ \diverg^y\bigl(a\,(\nabla^x w_4 - (\xi\cdot x)(\nabla^x w_3 + \nabla^y w_4)) \bigr) \nonumber\\
&\qquad\qquad+ \mu_0 w_3 + \mu_1 w_2
+ \mu_2 w_1-(\xi\cdot x) (\mu_0 w_2 +\mu_1 w_1 +\mu_2 w_0)\bigr] \nonumber \\ 
&+\  \eps^4 \,\bigl[\diverg^x\bigl(a\,(\nabla^x w_4 - (\xi\cdot x)(\nabla^x w_3 + \nabla^y w_4))\bigr) \\
&\qquad\qquad
- \diverg^y(a \,(\xi\cdot x) \nabla^x w_4) +\mu_0 w_4+\mu_1 w_3 +\mu_2 w_2\nonumber\\
&\qquad\qquad-(\xi\cdot x) (\mu_0 w_3 +\mu_1 w_2+ \mu_2 w_1)\bigr]\nonumber\\
&+\  \eps^5\, \bigl[-\diverg^x(a\, (\xi\cdot x)\nabla^x w_4 )+ \mu_1 w_4 +\mu_2 w_3\nonumber\\
&\qquad\qquad-(\xi\cdot x) (\mu_0 w_4 +\mu_1 w_3+ \mu_2 w_2)\bigr]\nonumber\\
&+\ \eps^6 \bigl[\mu_2 w_4 -(\xi\cdot x) (\mu_1 w_4 + \mu_2 w_3)\bigr] + \eps^7\bigl[-(\xi\cdot x) \mu_2 w_4\bigr].\nonumber 
\end{align}
Note that we omitted the arguments $x$ and $y$, $(y=x/\eps)$, to shorten the notation.

Owing to the regularity of the coefficients in the auxiliary problems~\eqref{homogenized_crossproblem}, 
\eqref{def:wbar1}, \eqref{def:what} and~\eqref{defbarw2}, which contribute to the definition of the terms $w_n$, $n = 1,\ldots, 4$ (see~\eqref{defw123} 
and \eqref{defw4}), and the fact that replacing $y$ by $x/\eps$ in a $Y$-periodic $L^2$-function originates a uniformly bounded function of $L^2(\omega)$, 
one can conclude that all the coefficients in the expansion~\eqref{eq:remainder}, except for those involving $\diverg^y$, are bounded in $L^2(\omega)$, 
independently of $\eps$. The terms involving $\diverg^y$ will give, after replacing $y$ by $x/\eps$, a sequence uniformly bounded in $H^{-1}(\omega)$.

In view of the continuous embedding $L^2(\omega)\subset H^{-1}(\omega)$, we may therefore infer
\begin{align}\label{estimate_crosssection}
 \normb{\diverg (a_\eps \beta_\eps \nabla w^\ast_{\eps}) + \mu^\ast_\eps \beta_\eps w^\ast_\eps}_{H^{-1}(\omega)} \leq C\eps^{3}.
\end{align}

{\textit{Step~2: Construction of a suitable compact and self-adjoint operator.}}
For fixed $\eps >0$, consider the linear and bounded operator
\begin{align*}
 L_\eps: H^{-1}(\omega)\to H_0^1(\omega),\qquad f\mapsto L_\eps f= w^{(\eps)}_f, 
\end{align*}
where the function $w_f^{(\eps)}$ is, by the Lax-Milgram theorem, the unique solution of  
\begin{align*}
-\diverg (a_\eps\beta_\eps \nabla w) + \beta_\eps w = f, &\qquad \text{ $w \in H_0^1(\omega)$.}
\end{align*}
The operator $L_\eps$ is an isomorphism from $H^{-1}(\omega)$ into $H_0^1(\omega)$.
Let $i$ denote the compact embedding
\begin{align*}
i: H^{1}_0(\omega)\to H^{-1}(\omega),\qquad \langle i(u), v\rangle_{H^{-1}(\omega)\times H^1_0(\omega)}=\int_\omega u v\dd{x}, 
\end{align*}
and define $\tilde L_\eps:= i\circ L_\eps$.
Notice that with this definition $\tilde L_\eps:H^{-1}(\omega)\to H^{-1}(\omega)$ is compact, but not self-adjoint. 
We introduce in $H^{-1}(\omega)$ the inner product $(\cdot,\cdot)_{\eps}^\sim$ given by
\begin{align*}
(f,g)_{\eps}^\sim =  (L_\eps f, L_\eps g)_{\eps}=\bigl( w^{(\eps)}_f, w^{(\eps)}_g\bigr)_{\eps},\qquad 
f,g\in H^{-1}(\omega),
\end{align*} 
where $(\cdot, \cdot)_{\eps}$ is defined as follows,
\begin{align*}
 (u,v)_{\eps}= \int_{\omega}a_\eps\beta_\eps \nabla u\cdot\nabla v + \beta_\eps uv\dd{x} , \qquad u,v\in H_0^1(\omega).
\end{align*}
The norm induced by $(\cdot, \cdot)_{\eps}$ on $H_0^1(\omega)$ is equivalent to the standard norm in $H_0^1(\omega)$
(with constants independent of 
$\eps$ and $s$). 
By ${H}^{-1}_{\eps}(\omega)$ we denote the Hilbert space $H^{-1}(\omega)$ endowed with the inner product $(\cdot, \cdot)_{\eps}^\sim$.
Then the operator $\tilde{L}_{\eps}:{H}^{-1}_{\eps}(\omega)\to {H}^{-1}_{\eps}(\omega)$ 
is both compact and self-adjoint. Indeed, for all $f,g\in H^{-1}_\eps(\omega)$, one has
\begin{align*}
 (\tilde{L}_{\eps} f, g)_{\eps}^\sim &=\bigl( w_{\tilde{L}_{\eps} f}^{(\eps)}, w_g^{(\eps)}\bigr)_{\eps} 
=\bigl\langle -\diverg (a_\eps\beta_\eps \nabla w_{\tilde{L}_{\eps} f}^{(\eps)}) + \beta_\eps w_{\tilde{L}_{\eps} f}^{(\eps)}, 
w_g^{(\eps)}\bigr\rangle_{H^{-1}(\omega)\times H^1_0(\omega)}\\
&=\int_\omega {L}_{\eps} f\; w^{(\eps)}_g\dd{x}
=\int_\omega w_f^{(\eps)} w^{(\eps)}_g\dd{x}.
\end{align*}
The right-hand side of the equation above is symmetric in $f$ and $g$. Hence, we may conclude the symmetry of the operator $\tilde{L}_{\eps}$. 
Compactness of $ \tilde{L}_{\eps}$ follows from the compactness of $i$. 

{\textit{Step~3: Applying Lemma~\ref{lem:VishikLusternik}.}}
In view of Step~2, the largest eigenvalue $\nu_\eps$ of the operator $\tilde{L}_{\eps}$ is related to the 
first eigenvalue $\mu_\eps$ of~\eqref{epsperturbed_crossproblem} 
by $\nu_\eps=\bigl(\beta_\eps(\mu_\eps + 1)\bigr)^{-1}$ for $\eps$ is sufficiently small, while 
for the corresponding eigenfunctions $f_\eps$ and $w_\eps$ one finds $f_\eps=-\diverg(a_\eps\beta_\eps\nabla w_\eps)+\beta_\eps w_\eps$.
Analogous relations hold for all eigenpairs of~\eqref{epsperturbed_crossproblem} in increasing order and of~$\tilde{L}_\eps$ in decreasing order.

Let $f^\ast_\eps:=-\diverg(a_\eps\beta_\eps\nabla w^\ast_\eps)+\beta_\eps w^\ast_\eps$ and 
$\nu^\ast_\eps:=\bigl(\beta_\eps(\mu^\ast_\eps+1)\bigr)^{-1}$ for $\eps$ small enough, where $\mu_\eps^\ast$ and $w_\eps^\ast$ 
are defined in \eqref{asymptotic_expansion}. 
Then, by \eqref{estimate_crosssection} and since the norm $||\cdot||^\sim_\eps$ induced by the inner product of 
$H^{-1}_\eps(\omega)$ is equivalent to the standard norm in $H^{-1}(\omega)$,
\begin{align*}
 \norm{\tilde{L}_{\eps} f^\ast_\eps-\nu^\ast_\eps f^\ast_\eps}_{\eps}^\sim
&= \nu^\ast_\eps\norm{\diverg (a_\eps\beta_\eps \nabla w^\ast_{\eps}) + \mu^\ast_\eps\beta_\eps w^\ast_\eps}_{\eps}^\sim \\
&\leq C \norm{\diverg (a_\eps\beta_\eps \nabla w^\ast_{\eps}) + \mu^\ast_\eps\beta_\eps w^\ast_\eps}_{H^{-1}(\omega)}
\leq \,C\eps^{3}.
\end{align*}
Classical homogenization guarantees that $\displaystyle \lim_{\eps\to 0}\mu_\eps^{(j)}=\mu^{(j)}_H$ for all $j\in \N_0$ with $\mu_H^{(j)}$ the $j$th eigenvalue of
\eqref{homogenized_crossproblem}, and since $\mu_H$ is simple, there
exists a constant $\alpha>0$, independent of $\eps$, satisfying $\mu^{(1)}_\eps-\mu_\eps \ge \alpha>0$ for all $\eps$ small enough 
(see, for instance, \cite[Chapter~11]{JikovKozlovOleinik94} and~\cite[Theorem~2.1]{KesaI79}). 
In particular, this implies 
\begin{align*}
 \displaystyle\lim_{\eps\to 0} \nu_\eps=(\mu_H + 1)^{-1}=\lim_{\eps\to 0} \nu^\ast_\eps \qquad\text{and} \qquad\nu_\eps-\nu_\eps^{(1)} \geq \alpha>0
\end{align*} 
with a constant $\alpha>0$, not depending on $\eps$, and $\nu_\eps^{(1)}$ denoting the second largest eigenvalue of $\tilde{L}_\eps$.
 
As a consequence, for $\eps>0$ sufficiently small, Lemma~\ref{lem:VishikLusternik} yields the estimates
\begin{align}\label{est:lf}
 \abs{\nu_\eps^\ast-\nu_\eps}\leq C\eps^{3} (\norm{f^\ast_\eps}_{\eps}^\sim)^{-1} 
\qquad\text{and}\qquad 
\norm{f_\eps^\ast-f_\eps}_{H^{-1}(\omega)}\leq C\norm{f_\eps^\ast-f_\eps}^\sim_{\eps}\leq C\eps^{3}.
\end{align}
Notice that in deriving~\eqref{est:lf} we took into account 
that both $\norm{f^\ast_\eps}_{\eps}^\sim$ and $(\norm{f^\ast_\eps}_\eps^\sim)^{-1}$ are uniformly bounded regarding $\eps$.
Indeed,~\eqref{asymptotic_expansion} in conjunction with the properties of $w_H$ entails the existence of constants $c, C>0$ such that
\begin{align*}
0<c \leq\norm{w_\eps^\ast}_\eps =\norm{f^\ast_\eps}_{\eps}^\sim  \leq C <\infty
\end{align*}
for all $\eps$ small enough.
Arguing that also $\nu_\eps^{-1}$ and $(\nu^\ast_\eps)^{-1}$ are 
uniformly bounded with respect to $\eps$, we obtain from the first inequality of~\eqref{est:lf} that
\begin{align*}
 \abs{\mu_\eps-\mu^\ast_\eps}\leq C\eps^{3}.
\end{align*}
On the other hand $(w_\eps -w^\ast_\eps)\in H^1 (\omega)$ solves the problem
\begin{align*}
\begin{cases}
-\diverg (a_\eps\beta_\eps \nabla w) + \beta_\eps w = (f_\eps - f^\ast_\eps) &  \text{in $\omega$,}\\
\hspace{3.1cm}w = - w^\ast_\eps  & \text{on $\partial \omega$}.
\end{cases}
\end{align*}
Due to the uniform ellipticity of the coefficients with respect to $\eps$ (see \eqref{uniformbounds} and the definition of $\beta_\eps$), we obtain
%
%
 \begin{align}\label{est:weps_wasteps}
 \norm{w_\eps - w^\ast_\eps}_{H^{1} (\omega)}\leq C \left(\norm{f_\eps- f^\ast_\eps}_{H^{-1}(\omega)}
+ \norm{w^\ast_\eps}_{H^{1/2}(\partial \omega)}\right).
\end{align}
We claim that
\begin{align}\label{est:boundary}
 \norm{w^\ast_\eps}_{H^{1/2}(\partial \omega)} \le C \eps^{1/2}.
\end{align}
In fact, \eqref{est:boundary} together with \eqref{est:lf}~and \eqref{est:weps_wasteps} implies the second estimate of our statement, concluding its proof.

The proof of the claim above follows along the lines of~\cite[Chapter~7.2]{CioDon99}. 
For the readers' convenience, we detail the proof of \eqref{est:boundary} in Appendix~\ref{appendix_B}.
\end{proof}
As a consequence of the previous lemma, we obtain the following result.

\begin{proposition}\label{prop:asymptotics_mu}
Under the hypothesis (H1) and if $\xi\in L^{\infty}((0, l);\R^2)$, the following convergence holds uniformly in $[0, \length]$:
\begin{align*}
 \lim_{\eps\to 0} { \frac{\mu_\eps  -\mu_H}{\eps^{2}}} = q_H + q_\xi\,,
\end{align*} 
with $q_H$ and $q_\xi$ defined in \eqref{def:q_H} and \eqref{def:q_xi}, respectively. Moreover, $w_\eps\to w_H$ in $L^2(Q_\length)$ as $\eps\to 0$, 
where $w_\eps(s,x)=w_\eps(s)(x)$ for $(s,x)\in Q_\length$.
\end{proposition}
The next lemma will provide additional regularity for the eigenfunctions $w_\eps(s)$ with respect to the variable $s$.

\begin{lemma}\label{lem:Lipschitz_continuity}
Assume that (H2) is satisfied and let $\eps>0$ be sufficiently small. Then the mappings $\mu_\eps: [0, \length]\to \R$, $s\to\mu_\eps(s)$ and 
$w_\eps:[0, \length]\to L^2(\omega)$, $s\mapsto w_\eps(s)$ are 
Lipschitz continuous with Lipschitz constant $C\eps$ and $C\sqrt{\eps}$, respectively.
\end{lemma}

\begin{proof}
Throughout this proof we fix a sufficiently small $\eps\in (0,1)$ and focus on the dependence on $s$. 
For the sake of simplicity we will omit some obvious dependences on $\eps$.

For $s\in [0,\length]$, consider the compact and self-adjoint operator 
\begin{align*}
 L(s): L^{2}(\omega)\to L^2(\omega),\qquad f\mapsto L(s) f= w_f(s), 
\end{align*}
where $w_f(s)$ is the unique solution of
\begin{align*}
-\diverg (a_\eps\beta_\eps(s) \nabla w) = \beta_\eps(s) f, &\qquad \text{ $w \in H_0^1(\omega)$,}
\end{align*}
to be understood as an element of $L^2(\omega)$ using the compact embedding $H_0^1(\omega)\subset\subset L^2(\omega)$.
The proof proceeds in three steps.

{\textit{Step 1: Lipschitz continuity of $L$.}}
We prove that there exists a constant $C>0$ such that for all $s, \bar{s}\in [0,\length]$ and $f\in L^2(\omega)$,
\begin{align}\label{est:K_Lipschitz}
 \norm{L(s)f-L(\bar{s})f}_{L^2 (\omega)}\leq C\eps \norm{f}_{L^{2}(\omega)}\,\abs{s-\bar{s}}.
\end{align}
Indeed, by the definition of $w_f(s)$ one has for $s,\bar{s}\in [0,\length]$ that
\begin{align*}
 -\diverg\bigl(a_\eps\beta_\eps(s) \nabla w_f(s)\bigr)&=\beta_\eps(s)f, \\
-\diverg\bigl(a_\eps\beta_\eps(\bar{s}) \nabla w_f(\bar{s})\bigr)&=\beta_\eps(\bar{s})f.
\end{align*}
Subtracting these two equations, testing the result with $\bigl(w_f(s)-w_f(\bar{s})\bigr)$ and integrating by parts gives 
\begin{align*}
& \int_\omega a_\eps \beta_\eps(s)\abs{\nabla w_f(s)- \nabla w_f(\bar{s})}^2 \dd{x}\nonumber\\
& = \int_{\omega} \bigl(\beta_\eps(s)-\beta_\eps(\bar{s})\bigr) \Bigl[a_\eps \nabla w_f(\bar{s})\cdot\nabla \bigl(w_f(\bar{s})-w_f(s)\bigr) + 
f\bigl(w_f(s)-w_f(\bar{s})\bigr)\Bigr]\dd{x}.\\
\end{align*}
Since (see \eqref{beta}) 
\begin{align*}
 \abs{\beta_\eps(s)-\beta_\eps(\bar{s})}\leq C\eps \abs{\xi(s)-\xi(\bar{s})}\leq C\eps \abs{s-\bar{s}}\qquad\text{for all $s, \bar{s}\in [0,\length],$}
\end{align*}
as a consequence of the Lipschitz continuity of $\xi$, and in view of
$\norm{\nabla w_f(\bar{s})}_{L^2(\omega)}\leq C\norm{f}_{L^{2}(\omega)}$ and the fact that $a_\eps$ is uniformly bounded, one obtains
\begin{align}\label{est:LipschitzL}
\int_\omega a_\eps \beta_\eps(s)\abs{\nabla w_f(s)- \nabla w_f(\bar{s})}^2 \dd{x}
\leq C\eps\abs{s-\bar{s}}\,\norm{w_f(s)-w_f(\bar{s})}_{H^1(\omega)}\norm{f}_{L^{2}(\omega)}.
\end{align}
On the other hand, in view of the uniform lower bound $a_\eps \beta_\eps(s)\geq c>0$ and Poincar\'{e}'s inequality, we find
\begin{align*}
\int_\omega a_\eps \beta_\eps(s)\abs{\nabla w_f(s)- \nabla w_f(\bar{s})}^2 \dd{x} 
\geq c\norm{w_f(s)- w_f(\bar{s})}^2_{H^1(\omega)},
\end{align*}
which, together with \eqref{est:LipschitzL}, entails \eqref{est:K_Lipschitz}.

{\textit{Step 2: Lipschitz continuity of the largest eigenpair of $L$.}} 
For $s\in [0,\length]$ let $\nu(s)$ denote the largest eigenvalue of $L(s)$. Notice that $\nu(s)=(\mu_\eps(s))^{-1}$. 
For the corresponding eigenfunction $f(s)$ we may assume that 
$\norm{f(s)}_{L^2(\omega)}=1$ and $f(s)>0$ in $\omega$.

Step~1 implies for $s, \bar{s}\in [0,\length]$ that
\begin{align}\label{est:Vishik}
\norm{L(\bar{s})f(s)-\nu(s)f(s)}_{L^2(\omega)} 
&=\norm{L(\bar{s})f(s)-L(s)f(s)}_{L^2(\omega)} 
\leq C\eps\abs{s - \bar{s}}.
\end{align}
Following Lemma~\ref{lem:VishikLusternik}, there exists an eigenvalue $\bar{\nu}(\bar{s})$ of $L(\bar{s})$ with
$\abs{\nu(s)-\bar{\nu}(\bar{s})}\leq C\eps \abs{s-\bar{s}}$.
If $\nu(\bar{s})\leq \nu(s)$, then $\bar{\nu}(\bar{s})\leq \nu(\bar{s})\leq \nu(s)$ and thus, 
$|{\nu}(s)-{\nu}(\bar{s})|\leq C\eps\abs{s-\bar{s}}$;
if $\nu(s)\leq \nu(\bar{s})$, the same result follows from exchanging the roles of $s$ and $\bar{s}$. Consequently, we have for all $s,\bar{s}\in [0,l]$ that
\begin{align}\label{est:nu}
 \abs{{\nu}(s)-{\nu}(\bar{s})}\leq C\eps \abs{s-\bar{s}}.
\end{align}
Let $\nu^{(1)}(s)$ represent the second eigenvalue in decreasing order of $L(s)$. In fact, $\nu^{(1)}(s)=\bigl(\mu^{(1)}_\eps(s)\bigr)^{-1}$. 
We claim that there exists $\alpha>0$ such that
\begin{align}\label{estimated}
d(s):=\nu(s)-\nu^{(1)}(s)\geq \alpha
\end{align}
for all $s\in [0,\length]$.

By \eqref{est:nu}, $\nu$ is continuous in $s$, and $-\nu^{(1)}$ is lower semicontinuous in $s$, provided $\mu^{(1)}_\eps$ is. Thus, if the latter is true,
$d$ attains its infimum in $[0,\length]$, which is strictly positive, since $\mu_\eps(s)$ is simple for every $s\in[0,\length]$. 

It remains to prove the lower semicontinuity of $s\mapsto \mu^{(1)}_\eps(s)$. For each $s\in[0,\length]$, $\mu_\eps(s)$ and $\mu^{(1)}_\eps(s)$ 
can be expressed through Rayleigh quotients as 
\begin{align*}
 \mu_\eps(s)=\frac{\int_\omega a_\eps\beta_\eps(s)\abs{\nabla w_\eps(s)}^2\dd{x}}{\int_\omega \beta_\eps(s)\abs{w_\eps(s)}^2\dd{x}}\qquad\text{and}\qquad
\mu^{(1)}_\eps(s)=\frac{\int_{\omega}a_\eps \beta_\eps(s)\abs{\nabla w^{(1)}_\eps (s)}^2\dd{x}}{\int_\omega \beta_\eps(s)\abs{w^{(1)}_\eps(s)}^2\dd{x}},
\end{align*}
respectively. Here, $w_\eps^{(1)}(s)\in H_0^1(\omega)$ is an eigenfunction for~\eqref{epsperturbed_crossproblem} regarding the second eigenvalue $\mu_\eps^{(1)}(s)$ such that
$\displaystyle \int_\omega \beta_\eps(s) |w^{(1)}_\eps(s)|^2\dd{x}=1$ and $\displaystyle \int_\omega \beta_\eps(s) w_\eps(s)w^{(1)}_\eps(s)\dd{x}=0$.
Let $\{s_n\}_n\subset[0,\length]$ and $s_0\in [0,\length]$ be such that $\displaystyle\lim_{n\to \infty} s_n=s_0$. 
Without loss of generality we may assume that $\displaystyle\liminf_{n\to \infty}\mu^{(1)}_\eps(s_n)<\infty$. Then $\{w_\eps(s_n)\}_n$ and $\{w^{(1)}_\eps(s_n)\}_n$ are 
uniformly bounded in $H^1_0(\omega)$. Hence, one can find two subsequences converging both weakly in $H^1(\omega)$ and strongly in $L^2(\omega)$ to limit functions
$w_0\in H_0^1(\omega)$ and $ w^{(1)}_0 \in H_0^1(\omega)$, respectively. In particular, $w_0$ and $w_0^{(1)}$ have unitary norms in $L^2(\omega)$ and
satisfy $\displaystyle\int_{\omega} w_0\ w_0^{(1)}\dd{x}=0$.
In view of the continuity of $\mu_\eps$ with respect to $s$, which results from \eqref{est:nu}, the above observations entail
\begin{align}\label{eq50}
 \mu_\eps(s_0)=\liminf_{n\to \infty}\mu_\eps(s_n)\geq \frac{\int_{\omega}a_\eps \beta_\eps(s_0)\abs{\nabla w_0}^2 \dd{x}}{\int_{\omega}\beta_\eps(s_0)\abs{w_0}^2\dd{x}} 
\geq \mu_\eps(s_0).
\end{align}
Notice that the last inequality follows from 
the min-max theorem.
Consequently, $w_0$ is the eigenfunction of~\eqref{epsperturbed_crossproblem} corresponding to $\mu_\eps(s_0)$, i.e.\ $w_0=w_\eps(s_0)$. 
Similarly to~\eqref{eq50}, we infer
\begin{align*}
 \liminf_{n\to \infty} \mu^{(1)}_\eps(s_n)\geq \frac{\int_\omega a_\eps \beta_\eps(s_0) \abs{\nabla w_0^{(1)}}^2\dd{x}}{\int_\omega \beta_\eps(s_0)\abs{w_0^{(1)}}^2\dd{x}}
\geq \mu^{(1)}_\eps(s_0),
\end{align*}
which finally is the stated lower semicontinuity of $\mu^{(1)}_\eps$ in $s$. Thus,~\eqref{estimated} is proven.

Let $s, \bar{s}\in [0, \length]$ with $\abs{s-\bar{s}}<\alpha/C$, where $C>0$ is the constant in \eqref{est:nu} and $\alpha>0$ is as in~\eqref{estimated}. Then, 
\begin{align}\label{est8}
 \abs{\nu(s)-\nu(\bar{s})}\leq \eps\alpha <\alpha.
\end{align}
Considering~\eqref{est:Vishik}, by Lemma~\ref{lem:VishikLusternik} there exists a linear combination $\bar{f}$ of eigenfunctions of $L(\bar{s})$ corresponding 
to the eigenvalues in 
$[\nu(s)-\sqrt{\eps}\alpha, \nu(s)+ \sqrt{\eps}\alpha]$ such that $\norm{\bar{f}}_{L^{2}(\omega)}=1$ and
\begin{align*}
 \norm{f(s)-\bar{f}}_{L^{2}(\omega)}\leq  (2C/\alpha)\sqrt{\eps}\abs{s-\bar{s}}.
\end{align*}
Since the interval above contains no eigenvalues of $L(\bar{s})$ other than $\nu(\bar{s})$ for $\eps$ sufficiently small (see~\eqref{estimated} and~\eqref{est8}), 
one finds
that $\bar{f}=f(\bar{s})$ (we may assume that $\bar{f}>0$). This proves 
\begin{align}\label{est:f}
 \norm{f(s)-f(\bar{s})}_{L^{2}(\omega)}\leq C\sqrt{\eps}\abs{s-\bar{s}}
\end{align}
for all $s, \bar{s}\in [0, \length]$ with a constant $C>0$, independent of $s$.

{\textit{Step 3: Lipschitz continuity of $\mu_\eps$ and $w_\eps$.}}
From~\eqref{est:nu} together with the continuity of $s\mapsto\mu_\eps(s)$ on the compact interval $[0,\length]$ and the uniform boundedness of 
$\mu_\eps$ regarding $\eps$
one infers, representing by $C$ a constant with respect to $s$,
 \begin{align*}
 \abs{{\mu_\eps}(s)-{\mu}_\eps(\bar{s})}\leq C\eps\abs{s-\bar{s}} \abs{\mu_\eps(s)}\,\abs{\mu_\eps(\bar{s})} \leq C\eps \abs{s-\bar{s}}
 \end{align*}
for all $s, \bar{s}\in [0,\length]$. This proves the stated Lipschitz continuity of $\mu_\eps$. 
Finally, \eqref{est:K_Lipschitz}~and~\eqref{est:f} lead to
\begin{align*}
\norm{w_\eps(s)-w_\eps(\bar{s})}_{L^2(\omega)}&=\norm{w_{f(s)}(s)-w_{f(\bar{s})}(\bar{s})}_{L^2(\omega)}=\norm{L(s)f(s)-L(\bar{s})f(\bar{s})}_{L^2(\omega)}\\
&\leq \norm{L(s)f(s)-L(\bar{s})f(s)}_{L^2(\omega)}+\norm{L(\bar{s})(f(s)-f(\bar{s}))}_{L^2(\omega)}\\
&\leq C\eps \abs{s-\bar{s}} + C\norm{f(s)-f(\bar{s})}_{L^{2}(\omega)}\leq C\sqrt{\eps} \abs{s-\bar{s}}
\end{align*}
for all $s, \bar{s}\in [0,\length]$.
This completes the proof.
\end{proof}

\subsubsection{Variational formulation}
The asymptotic expansion of the cross section problem~\eqref{epsperturbed_crossproblem} (see~\eqref{asymptotic_expansion} and Lemma~\ref{lem:asymptotics_mueps}) 
suggests that, in order to obtain the desired spectral convergence, one needs to subtract from $\tilde{E}_{\eps, \eps}$, defined in \eqref{epschangedenergy}, 
the quadratic term
\begin{align*}
 \int_{Q_\length}\beta_\eps\frac{\mu_H}{\eps^{2}} \abs{v}^2\dd{s}\dd{x}. 
\end{align*}
Precisely, we will consider the functionals $E_{\eps}:L^2(Q_\length)\to \overline{\R}$ with $\eps>0$ given by
\begin{align}\label{def:Falphaeps} 
E_{\eps}[v]= \tilde{E}_{\eps, \eps}[v] - \int_{Q_\length}\beta_\eps\frac{\mu_H}{\eps^{2}} \abs{v}^2\dd{s}\dd{x} 
\end{align}
for $v\in H^1_0(Q_\length)$ and $E_\eps[v]=\infty$ otherwise.
\subsection{The torsion-free case}
First, we discuss the case without rotation of the cross section regarding the Tang frame, i.e.\ $\tau=0$, which is substantially easier than 
dealing with non-vanishing torsion. 
In fact, it can be treated by
using only the expansion of the cross section problem (see Section~\ref{subsec:asymptotics}) as opposed 
to the expansion of the full problem, which is inevitable when 
considering twisted waveguides (see Section~\ref{sec:propagation}).

\subsubsection{$\Gamma$-convergence of the energies $\{E_\eps\}_\eps$}\label{subsec:Gammaconvergence_Eeps}
The strategy for characterizing the asymptotic behavior of \eqref{rescaledEVP}-\eqref{EVPepsilondelta_notorsion} 
with $\delta=\eps$ is to apply Lemma~\ref{theo:Gammaconvergence} 
to the sequence of functionals $\{E_\eps\}_\eps$. 
 
\begin{proposition}\label{prop:Gammaconvergence_Heps}
If hypotheses (H1)-(H2) are satisfied and if $\tau=0$, the sequence $\{E_\eps\}_\eps$ defined in \eqref{def:Falphaeps} meets the conditions $(i)$ and 
$(ii)$ of Lemma~\ref{theo:Gammaconvergence} and $\Gamma$-converges, with respect to the strong topology in $L^2(Q_\length)$, as $\eps \to 0$, 
to the functional $E_0$ given through
 \begin{align*}
 E_0[v]=\left\{\begin{array}{ll}
\displaystyle \int_0^\length \bar{a}\abs{\varphi'}^2 + (q_{H}+ q_{\xi}) \abs{\varphi}^2\dd{s}, & 
\!\!\text{if $v(s,x)=w_H(x) \varphi(s), (s,x) \in Q_\length, \varphi\in H_0^1(0,\length)$,}\\
+\infty, & \text{otherwise},
 \end{array}\right.
 \end{align*} 
with $\bar{a}=\int_Y a\ d{y}$ and $q_{H}$ and  $q_{\xi}$ defined in \eqref{def:q_H} and \eqref{def:q_xi}, respectively.
\end{proposition}

\begin{proof} We split the proof into five steps. Step 1 and  Step 2 are dedicated to the proof of conditions $(i)$ and $(ii)$. 
In Step 3 we characterize the limit of the bounded sequences with bounded energies. Finally, in Step 4 and Step 5, 
we prove the $\Gamma$-convergence of the sequence $\{E_\eps\}_\eps$.

{\it{Step~1: Lower bound.}} Recall that $(\mu_\eps(s), w_\eps(s))$ denotes the first normalized eigenpair of
\eqref{epsperturbed_crossproblem} for $s\in [0,\length]$ and $\eps>0$. Besides, the variational formula for the first eigenvalue entails the representation
\begin{align}\label{minmax_mu0}
 \mu_\eps(s) = \inf_{w\in H_0^1(\omega), w\neq 0} 
\frac{\int_{\omega} a_\eps \beta_\eps(s)\abs{\nabla w}^2\dd{x}}{\int_\omega\beta_\eps(s) \abs{w}^2\dd{x}}
=\int_\omega a_\eps\beta_\eps(s) \abs{\nabla w_\eps(s)}^2\dd{x},
\end{align}
where we used~\eqref{normalization}.
From \eqref{minmax_mu0} one derives for $v\in H_0^1(Q_\length)$,
\begin{align*}
E_{\eps}[v]\geq \frac{1}{\eps^{2}} \int_{Q_\length}a_\eps\beta_\eps\, \abs{\nabla^x v}^2 - \mu_H\beta_\eps\abs{v}^2 \dd{s} \dd{x}
\geq \int_0^\length \frac{\mu_\eps - \mu_H}{\eps^{2}} \Bigl(\int_{\omega} \beta_\eps \abs{v}^2\dd{x}\Bigr)\dd{s}.
\end{align*}
As $\displaystyle \lim_{\eps\to 0}\int_{\omega} \beta_\eps(s) \abs{v(s,\cdot)}^2\dd{x}=\norm{v(s, \cdot)}_{L^2(\omega)}^2$ uniformly in $s$, it follows from 
Proposition~\ref{prop:asymptotics_mu} that
\begin{align*}
 \int_0^\length \frac{\mu_\eps - \mu_H}{\eps^{2}} \int_{\omega} \beta_\eps \abs{v}^2\dd{x}\dd{s} \to \int_0^\length (q_{H}+q_\xi)\, \norm{v}_{L^2(\omega)}^2\dd{s}
\end{align*}
for $\eps\to 0$.
Since $\xi\in W^{1,\infty}((0,\length);\R^2)$ by \textit{(H2)}, the $L^\infty(0, \length)$-norm of $q_{\xi}$ 
is bounded, which guarantees the existence of a constant $c>0$ such that $E_\eps[v]\geq -c\norm{v}^2_{L^2(Q_\length)}$ 
for all $v\in H_0^1(Q_\length)$ and all $\eps$ sufficiently small.

{\it{Step~2: Compactness of sequences with bounded energy.}} Let $\{v_\eps\}_\eps$ be a bounded sequence in $L^2(Q_\length)$ such that
\begin{align*}
 \sup_\eps {E_\eps[v_\eps]}\leq C<\infty. 
\end{align*}
Exploiting the structure of the second summand of $\tilde{E}_{\eps, \eps}$, in particular the fact that in view 
of~\eqref{uniformbounds} and \eqref{beta}, $a_\eps \beta_\eps\geq c>0$ uniformly in $s$ for $\eps$ sufficiently small,
yields $\norm{\nabla^x v_\eps}_{L^2(Q_\length)}\leq C <\infty$. On the other hand, since 
$a_\eps/\beta_\eps\ge c>0$, 
independently of $s$ and $\eps$, for $\eps$ small enough, we derive from the first term in $\tilde{E}_{\eps, \eps}$ with $\tau=0$ by using the lower 
bound of Step~1 that $\norm{v_\eps'}_{L^2(Q_\length)}\leq C <\infty$. 

Hence, $\{v_\eps\}_\eps$ is bounded in 
$H^1_0(Q_\length)$ uniformly with respect to $\eps$ and we infer the existence of a $v\in H_0^1(Q_\length)$ such that, 
up to a subsequence, $v_\eps \weakly v$ in $H^1_0(Q_\length)$ and
by compact embedding $v_\eps\to v$ in $L^2(Q_\length)$. 

{\it{Step~3: Separation of variables.}} We assert that the limit function $v$ is of the form 
\begin{align}\label{splitting}
 v(s, x)=w_H(x) \varphi(s), \qquad \text{$(s,x)\in Q_\length$,}
\end{align}
with $\varphi\in H_0^1(0,\length)$.
Indeed, if we consider a bounded energy sequence $\{v_\eps\}_\eps$ as in Step~2, we get 
\begin{align}\label{estimate_splitting}
  0 &\leq \int_0^\length\int_{\omega} Q \nabla^x v\cdot \nabla^x v -\mu_H \abs{v}^2 \dd{x}\dd{s}\nonumber\\
&\leq \int_0^\length \liminf_{\eps\to 0}\Bigl(\int_{\omega} a_\eps \,\abs{\nabla^x v_\eps}^2\dd{x}\Bigr)\dd{s} 
 - \lim_{\eps\to 0}\int_{Q_l} \mu_H \,\abs{v_\eps}^2\dd{s}\dd{x} \nonumber\\
&\leq \liminf_{\eps\to 0} \int_{Q_\length} a_\eps \beta_\eps\,\abs{\nabla^x v_\eps}^2 - \mu_H \beta_\eps\abs{v_\eps}^2\dd{s}\dd{x} 
\leq \lim_{\eps\to 0}\,C\eps^{2}=  0.
\end{align}
The first inequality is a consequence of $\mu_H$ being the first eigenvalue of \eqref{homogenized_crossproblem}, 
and the second estimate follows from a classical homogenization result that guarantees, for almost every $s\in (0,l)$, the $\Gamma$-convergence of the functional 
$\displaystyle \int_\omega a_\eps |\nabla^x v|^2\dd{x}$ to $\displaystyle \int_\omega Q \nabla^x v \cdot \nabla^x v\dd{x}$ 
with $Q$ defined in~\eqref{definitionQ}
(see for instance~\cite[Chapter 5]{JikovKozlovOleinik94}) . 

Since the left-~and right-hand side of~\eqref{estimate_splitting} coincide, all the inequalities turn into equalities and
$v$ satisfies
\begin{align}\label{characterizing_forumla_u0}
 \int_0^\length\int_{\omega} Q \nabla^x v \cdot \nabla^x v - \mu_H\abs{v}^2\dd{x}\dd{s}=0.
\end{align} 
Accounting for the Rayleigh quotient representation of $\mu_H$, we find that 
\begin{align*}
s\mapsto \int_{\omega} Q\nabla^x v(s,x) \cdot \nabla^x v(s, x) -\mu_H \abs{v(s,x)}^2\dd{x}
\end{align*}
is non-negative. Therefore, due to \eqref{characterizing_forumla_u0} this function 
vanishes almost everywhere in $[0,\length]$. Using the minimum formula for $\mu_H$ entails the existence of a function $\varphi:[0,\length]\to\R$ such that
$v(s, x)=w_H(x)\varphi(s)$ for almost all $(s,x)\in Q_\length$. We know that $v \in H_0^1(Q_\length)$ and 
$w_H\in C^\infty(\bar{\omega})\cap H_0^1(\omega)$. Thus, $\varphi\in H_0^1(0,\length)$.   
 
{\it{Step~4: Liminf-inequality.}}  
Let $v_\eps\to v$ in $L^2(Q_\length)$ and assume without loss of generality that 
$\lim_{\eps \to 0}E_{\eps}[v_\eps]=\liminf_{\eps\to 0} E_\eps[v_\eps]< \infty$. 
By Steps~2 and~3, $\{v_\eps\}_\eps\subset H_0^1(Q_\length)$ is uniformly bounded in $H^1_0(Q_\length)$ and $v$ is of the form \eqref{splitting}. We define for every 
$\eps>0$ the auxiliary function $z_\eps=v_\eps-v= v_\eps-w_H\varphi$. 
Notice that by construction $\norm{z_\eps}_{H^1_0(Q_\length)}\leq C$ for all $\eps>0$ and $z_\eps\to 0$ in $L^2(Q_\length)$.
Together with Step~1 it follows that  
\begin{align*}
&\lim_{\eps\to 0} E_{\eps}[v_\eps] \geq \liminf_{\eps\to 0}\int_{Q_\length}  \frac{a_\eps}{\beta_\eps} \,\abs{v_\eps'}^2\dd{s}\dd{x} 
+ \int_0^\length (q_{H}+q_{\xi})\norm{v}^2_{L^2(\omega)}\dd{s}\\
&\qquad= \liminf_{\eps\to 0} \int_{Q_\length}   a_\eps \,\abs{w_H\varphi' + z_\eps'}^2\dd{s}\dd{x}
+ \int_0^\length (q_{H}+q_{\xi})\norm{w_H}_{L^2(\omega)}^2\abs{\varphi}^2\dd{s}\\
&\qquad= \liminf_{\eps\to 0}\Bigl(\int_{\omega}  a_\eps\, \abs{w_H}^2\dd{x}\Bigr)\Bigl(\int_0^\length\abs{\varphi'}^2\dd{s}\Bigr)  + \int_{Q_\length} 
a_\eps
\,\abs{z_\eps'}^2\dd{s}\dd{x} \\
&\qquad \qquad \qquad \qquad\qquad \qquad \qquad \qquad
+ 2\int_{Q_\length} \left( a_\eps z_\eps'\right) (w_H\varphi')\dd{s}\dd{x}
+ \int_0^\length (q_{H}+q_{\xi})\abs{\varphi}^2\dd{s}\\
&\qquad\geq \bar{a} \int_0^\length \, \abs{\varphi'}^2 \dd{s}+ \int_0^\length (q_{H}+q_{\xi})\abs{\varphi}^2\dd{s}= E_0[v].
\end{align*}
In the estimate above we used $\beta_\eps\to 1$ uniformly in $Q_\length$, $\norm{w_H}_{L^2(\omega)}=1$, $a_\eps\weaklystar\bar{a}$ in $L^\infty(\omega)$ 
and, since $a_\eps$ does not depend on $s$, $a_\eps z_\eps'\weakly 0$ in $L^2(Q_\length)$ for $\eps\to 0$.
Indeed, from the uniform boundedness of $\{a_\eps z_\eps'\}_\eps$ in $L^2(Q_\length)$ we infer 
that, after passing to a subsequence, $ a_\eps z_\eps'\weakly z_0$ in $L^2(Q_\length)$ for some $z_0\in L^2(Q_\length)$. 
On the other hand, $a_\eps z_\eps\weakly 0$ in $L^2(Q_\length)$ yields $a_\eps z_\eps'= 
( a_\eps z_\eps)' \weakly 0$ in $H^{-1}(Q_\length)$. 
Hence, a comparison of the limits gives $z_0=0$ and we may conclude weak convergence of the full sequence $\{a_\eps z_\eps'\}_\eps$.

{\it{Step~5: Recovery sequence.}} Let $v\in L^2(Q_\length)$ be of the form $v=w_H\varphi$ with
$\varphi\in H_0^1(0,\length)$. 
We define the sequence $\{v_\eps\}_\eps\subset H_0^1(Q_\length)$ by setting
\begin{align*}
 v_\eps(s,x)=w_\eps(s,x)\varphi(s), \qquad\text{$(s,x)\in Q_\length$},
\end{align*}
where for each $s\in [0, l]$ and $\eps>0$, $\ w_\eps(s, \cdot)$ is the first eigenfunction of~\eqref{epsperturbed_crossproblem}. By construction,
Proposition~\ref{prop:asymptotics_mu} implies $v_\eps\to v$ in $L^2(Q_l)$.

From Lemma \ref{lem:Lipschitz_continuity} we have that $s\mapsto w_\eps(s)$ is a Lipschitz continuous function from $[0, l]$ into $L^2(\omega)$, 
with Lipschitz constant $C\sqrt\eps$. Then, $w_\eps(s)$ is almost everywhere differentiable in $s$ and 
$\displaystyle \|w_\eps' (s)\|_{L^2(\omega)}\le C\sqrt\eps$ for all $s\in [0,l]$. Consequently, as $\eps\to 0$,
\begin{align}\label{sderivative}
\int_{Q_l} |w_\eps' (s, x)|^2\dd{s}\dd{x}\to 0.
\end{align}
Since $w_\eps\to w_H$ in $L^2(Q_l)$ by Proposition~\ref{prop:asymptotics_mu}, and 
$a_\eps\weaklystar\bar a$ in $L^{\infty}(\omega)$ and therefore also in $L^\infty(Q_l)$, we obtain together with~\eqref{sderivative} that
\begin{align*}
\limsup_{\eps\to 0} \int_{Q_\length} a_\eps \abs{v_\eps'}^2\dd{s}\dd{x} &
=\limsup_{\eps\to 0} \int_{Q_\length} a_\eps \abs{w_\eps}^2\, |\varphi'|^2\dd{s}\dd{x} \\ &= \bar a \norm{w_H}^2_{L^2(\omega)} \int_0^l |\varphi'|^2\dd{s}=\bar a\int_0^l |\varphi'|^2\dd{s}.
\end{align*}
Then, using again Proposition~\ref{prop:asymptotics_mu}, in combination with $\beta_\eps\to 1$ uniformly in $Q_l$ and~\eqref{normalization}, gives
\begin{align*}
&\limsup_{\eps\to 0} E_{\eps}[v_\eps]
=\limsup_{\eps\to 0} \int_{Q_\length} \frac{a_\eps}{\beta_\eps} \, \abs{v_\eps'}^2\dd{s}\dd{x}\\
&\qquad \qquad\qquad\qquad\qquad\qquad\qquad+ \limsup_{\eps \to 0} \frac{1}{\eps^{2}} \int_0^\length\Bigl(\int_\omega a_\eps \beta_\eps \abs{\nabla^x w_\eps}^2 
- \mu_H\beta_\eps\abs{w_\eps}^2 \dd{x}\Bigr) \abs{\varphi}^2 \dd{s}\\
&\quad=\limsup_{\eps\to 0} \int_{Q_\length} a_\eps \, \abs{v_\eps'}^2\dd{s}\dd{x}
+ \int_0^\length\limsup_{\eps\to 0} \frac{\mu_\eps-\mu_H}{\eps^{2}} 
\,\Bigl(\int_\omega \beta_\eps\abs{w_\eps}^2 \dd{x} \Bigr)\,\abs{\varphi}^2 \dd{s}\\
&\quad=\bar{a}\int_0^\length \abs{\varphi'}^2\dd{s}+\int_0^\length (q_H + q_\xi)\ \abs{\varphi}^2\dd{s}=E_0[v],
\end{align*}
which shows the required limsup-inequality and concludes the proof of Proposition~\ref{prop:Gammaconvergence_Heps}.
\end{proof}

\subsubsection{Statement and proof of the main result}\label{subsec:mainresult}
Finally, we can formulate our main theorem capturing the full asymptotics 
of~\eqref{rescaledEVP}-\eqref{EVPepsilondelta_notorsion} for 
$\delta=\eps$ under the assumption of vanishing torsion. 
For this purpose, we recall that $(\mu_H, w_H)$ is the first eigenpair of the homogenized 
cross section problem~\eqref{homogenized_crossproblem}
and set $(\eta_P^{(j)}, \varphi_P^{(j)})$ to be the $j$th eigenpair of 
the Sturm-Liouville eigenvalue problem
 \begin{align*}
  -\bar{a}\varphi'' + q \varphi =\eta\varphi, \qquad \varphi\in H_0^1(0,\length),
 \end{align*}
where $q=q_{H} + q_\xi$ with $q_{H}$ and $q_\xi$ defined in~\eqref{def:q_H} and~\eqref{def:q_xi}, respectively.

\begin{theorem}\label{theo:mainresult_cylinder}
Suppose that the hypotheses (H1)-(H2) are satisfied and that $\tau=0$. For $j\in \N_0$, 
let $\{(\lambda_\eps^{(j)}, u_\eps^{(j)})\}_\eps$ be a sequence of $j$th 
eigenpairs for the spectral problem~\eqref{rescaledEVP}-\eqref{EVPepsilondelta_notorsion} with $\delta=\eps$.
Then, for every $\eps>0$ one has 
\begin{align*}
\lambda_\eps^{(j)}= \frac{\mu_H}{\eps^{2}} + \eta_{\eps}^{(j)},\qquad\text{where }
\lim_{\eps\to 0}\eta^{(j)}_\eps=\eta_P^{(j)},
\end{align*}
and the sequence of eigenfunctions $\{u_\eps^{(j)}\}_\eps$ converges, up to a subsequence, in the following sense: 
\begin{align*}
 u_\eps^{(j)}\circ \psi_{\eps}=v_\eps^{(j)}\quad\longrightarrow\quad 
v^{(j)}:=w_H \,\varphi^{(j)}_P \qquad \text{in $L^2(Q_\length)$ as $\eps \to 0$.}
\end{align*} 
Here, $\psi_{\eps}$ is the parameter transformation introduced in~\eqref{rescaling}. 

Conversely, any such $v^{(j)}$ is the $L^2(Q_\length)$-limit of a sequence $\{u_\eps^{(j)}\circ \psi_{\eps}\}_\eps$ 
with $u_\eps^{(j)}$ an eigenfunction of~\eqref{rescaledEVP}-\eqref{EVPepsilondelta_notorsion} corresponding to the eigenvalue $\lambda_\eps^{(j)}$. 
\end{theorem}
\begin{proof}
To conclude we simply need to join the previous results of Sections~\ref{subsec:Gammaconvergence_Eeps} and~\ref{subsec:Gamma} together.
In fact, Proposition~\ref{prop:Gammaconvergence_Heps} allows us to apply Lemma~\ref{theo:Gammaconvergence} to $\{E_\eps\}_\eps$ with $H_\eps=L^2(Q_l)$
endowed with the inner product $(u,v)_\eps:=\int_{Q_l}\beta_\eps u v\dd{x}$ for $u,v\in L^2(Q_l)$. We recall that $\beta_\eps$ given by~\eqref{beta} 
converges uniformly to $1$.
As a consequence, the eigenpairs of 
\begin{align*}
 \Bcal_\eps:=\Acal_{\eps, \eps}-\mu_H\eps^{-2}
\end{align*}
(see~\eqref{EVPepsilondelta_notorsion} for the definition of $\Acal_{\eps, \eps}$)
converge to those of a limit operator $\Bcal_0$.
Owing to the structure of the $\Gamma$-limit functional $E_0$ in Proposition~\ref{prop:Gammaconvergence_Heps} the operator 
$\Bcal_0:L^2(Q_\length)\to L^2(Q_\length)$ is given through
\begin{align*}
\Bcal_0 v= (-\bar{a}\varphi''+ q \varphi)w_H, \qquad \text{$v\in \Dcal(\Bcal_0)$,}   
\end{align*}
with $\Dcal(\Bcal_0)=\{v\in L^2(Q_\length): v=w_H\varphi,\, \varphi\in H_0^1(0,\length)\}$.
Thus, the proof of Theorem~\ref{theo:mainresult_cylinder} is complete.
\end{proof}

\subsection{The general case with torsion}\label{sec:propagation}
If the cross section rotates relatively to the classical Tang frame, i.e.~$\tau\neq 0$, an expansion of the full eigenvalue problem 
(see~\eqref{rescaledEVP}-\eqref{EVPepsilondelta} with $\delta=\eps$)
\begin{align}\label{full_eigenvalueproblem}
 \Acal_{\eps, \eps} v_\eps
=\lambda_\eps \beta_\eps v_\eps, 
\qquad v_\eps\in H_0^1(Q_\length),
\end{align}
is necessary to capture its asymptotic 
behavior as $\eps$ tends to zero.

Performing an asymptotic expansion of~\eqref{full_eigenvalueproblem} with the ansatz 
\begin{align}\label{ansatz:full_asymptotic_expansion}
v^\ast_\eps(s,x)&= v_0(s,x,y)+\eps v_1(s,x,y)+\eps^2 v_2(s,x,y)+\eps ^3 v_3(s,x,y)+\eps^4 v_4(s,x,y),\quad y=x/\eps,\nonumber\\
\lambda^\ast_\eps(s)&= \eps^{-2}\lambda_{-2} + \eps^{-1}\lambda_{-1} + \eps\lambda_0,\\
\beta_\eps(s,x)&= 1-\eps(\xi(s)\cdot x), \qquad
\beta^{-1}_\eps(s,x)= 1+\eps (\xi(s)\cdot x) + \eps^2(\xi(s)\cdot x)^2+\ldots.\nonumber
\end{align}
entails
\begin{align*}
 \lambda_{-2}=\mu_H, \qquad \lambda_{-1}=0,\qquad \lambda_0=\eta_P.
\end{align*}
Generalizing the definition in Section~\ref{subsec:mainresult}, we denote by $(\eta_P, \varphi_P)$ the 
first eigenpair of the Sturm-Liouville problem
\begin{align}\label{EVP_1d_torsion}
 -\bar{a}\varphi''+q\varphi=\eta\varphi, \qquad\varphi\in H_0^1(0,\length),
\end{align}
with the potential 
\begin{align}\label{def:q_torsion}
 q=q_\tau + q_{H}+q_{\xi},
\end{align}
where $q_{H}, q_{\xi}$ are defined in~\eqref{def:q_H},~\eqref{def:q_xi}, respectively, and
\begin{align*}
q_\tau(s):= \tau(s)^2
T_{ijkl}\left(\int_{\omega}(Rx)_i\partial_j^x w_H(x) (Rx)_k \partial_l^x w_H(x)\dd{x} \right), \qquad s\in [0,l].
\end{align*}
For $i,j,k,l\in\{1,2\}$, 
\begin{align*}
 T_{ijkl}:= Q_{ij}\delta_{kl}  +  \int_{Y}a \partial_l^y \vartheta_{ijk} \dd{y},
\end{align*}
where $\vartheta=(\vartheta_{ijk})\in H^1_\#(Y;\R^{8})$ has mean value zero with respect to the unit cell and solves
\begin{align*}
 -\diverg^y(a\nabla^y \vartheta_{ijk})=\partial_k^y (a\partial_j^y \phi_i) + \delta_{ij}\partial_k^y a.
\end{align*}
Making use of the expansion for the cross section problem in Section~\ref{subsubsec:formal_expansion}, we determine the coefficients of the eigenmodes as
\begin{align*}
 v_0(s,x,y) &= w_0(s,x,y)\varphi_P(s)=w_H(x)\varphi_P(s),\nonumber\\
 v_1(s,x,y) &= w_1(s,x,y)\varphi_P(s),\\
 v_2(s,x,y) &= w_2(s,x,y)\varphi_P(s) + \bar{v}_2(s,x),\nonumber\\
 v_3(s,x,y) &= w_3(s,x,y)\varphi_P(s) + \bar{v}_3(s,x). \nonumber
\end{align*}
Here,
\begin{align*}
 \bar{v}_3(s,x)&=\tau(s)^2\vartheta_{ijk}(y)(Rx)_j(Rx)_k\partial_i^x w_H(x) \varphi_P(s)\nonumber\\
&\qquad\qquad\qquad\qquad\qquad+  \tau(s)\phi_i(y)(Rx)_iw_H(x)\varphi'_P(s)+\phi_i(y)\partial_i^x\bar{v}_2(s,x),
\end{align*}
and $\bar{v}_2(s,x)= \bar{w}_{21}(s,x)\varphi_P(s) + \bar{w}_{22}(s,x)\varphi'_P(s)$.
For fixed $s\in [0,\length]$ the functions $\bar{w}_{21}(s,\cdot)$ and $\bar{w}_{22}(s, \cdot)$ are orthogonal to $w_H$ regarding the $L^2(\omega)$-inner 
product and solve
\begin{align*}
 -Q_{ij} \partial_{ij}^x \bar{w}_{21} -\mu_H \bar{w}_ {21} = q_\tau w_H + \tau^2 T_{ijkl} (Rx)_k(Rx)_j\partial_{il}^x w_H 
+ \tau' Q_{ij} (Rx)_j\partial_i^x w_H
\end{align*}
and 
\begin{align*}
  -Q_{ij} \partial_{ij}^x \bar{w}_{22} -\mu_H \bar{w}_{22} = 2 \tau Q_{ij} (Rx)_j \partial_i^x w_H
\end{align*}
in $H_0^1(\omega)$, respectively. The coefficient of order four in \eqref{ansatz:full_asymptotic_expansion} takes the form $v_4=w_4\varphi_P+\bar{v}_4$.
We give the explicit expressions of $w_4$ and $\bar{v}_4$ in Appendix~\ref{appendix_A1} and~\ref{appendix_A2}.

Based on the results of the full asymptotic expansion above, we conjecture that the statement 
of Theorem~\ref{theo:mainresult_cylinder} for the case without torsion remains true if $q$ is given by~\eqref{def:q_torsion} and $(\eta^{(j)}_P, \varphi^{(j)}_P)$ is the 
$j$th eigenpair of~\eqref{EVP_1d_torsion}.
At this point, however, a rigorous argument is still missing. The difficulties lie in the proof of the lower bound for a $\Gamma$-convergence result
in the spirit of Proposition~\ref{prop:Gammaconvergence_Heps}, and are due
to a lack of appropriate compactness.
\begin{remark}
Considering a tube with homogeneous cross section, i.e.\ $a\equiv 1$, the formulas for the potential 
$q=q_\tau +q_H + q_\xi$ simplify to
\begin{align*}
q_H =0,\qquad q_\xi =-\frac{1}{4} \abs{\xi}^2,\qquad q_{\tau}=\tau^2 \int_{\omega}\abs{\nabla w_H\cdot Rx}^2\dd{x}.
\end{align*}
This is in agreement with~\cite{BMT07}. 
\end{remark}

\section{Waveguides with an inhomogeneous cross section}\label{sec:case_inhomogeneous}

This section covers the regime $\eps=1$, so that for $\delta>0$ the energy functional \eqref{changedenergy} reads
 \begin{align*}
 \tilde{E}_{1,\delta} [v]= \int_{Q_\length} \frac{a}{\beta_\delta} \abs{v' + \tau (\nabla^x v\cdot Rx)}^2 
+ \frac{a \beta_\delta}{\delta^2} \abs{\nabla^x v}^2\dd{s}\dd{x},
 \end{align*}
with $a\in L^\infty(\omega)$ such that $a(x)\geq c>0$ for almost every $x\in \omega$.

\subsection{Asymptotic analysis of the cross section eigenvalue problem}\label{subsec:asymptotics_inhomogeneous} 
In analogy to Section~\ref{sec:case_oscillatingcoefficients}, we start by studying the cross section problem
\begin{align}\label{cross_sec_perturb}
 -\diverg(a \beta_\delta(s)\nabla w)=\mu\beta_\delta(s) w, \qquad w\in H^1_0(\omega),
\end{align}
for $s\in [0,\length]$, the first (normalized) eigenpair of which is denoted by $(\mu_\delta(s), w_\delta(s))$. 
Due to the lack of a fast variable there is no need for a two-scale ansatz when emulating the formal computations of the previous section. 
A simple one-scale asymptotic expansion of \eqref{cross_sec_perturb} up to second order yields
\begin{align}\label{expansion_inhomogeneous} 
\mu_\delta^\ast(s) &= \mu_0(s) + \delta \mu_1(s) + \delta^2 \mu_2(s) = \mu_C + \delta (b\cdot \xi(s)) + \delta^2 q_{\rm c}(s),\nonumber\\
 w_\delta^\ast(s,x) &= w_0(s,x) + \delta w_1(s,x) + \delta^2 w_2(s,x) \\ 
&= w_C(x) + \delta \hat{w}(x)\cdot \xi(s) + \delta^2 \bar{w}(x)\xi(s)\cdot \xi(s), \qquad (s,x)\in Q_\length.\nonumber
\end{align}
Let us now explain the quantities involved. By $(\mu_C, w_C)$ we denote the first eigenpair of the cross section problem
\begin{align}\label{cross_sec}
 -\diverg(a \nabla v)=\mu v, \qquad v\in H^1_0(\omega),
\end{align}
We assume that $w_C$ is normalized, i.e.\ $\int_{Q_\length}w_C^2\dd{x}=1$.
While $\xi$ was defined in \eqref{beta}, the vector $b$ represents 
\begin{align}\label{def:b}
 b:=\int_\omega a \nabla w_C w_C\dd{x},
\end{align}
and $\hat{w}\in H^1_0(\omega;\R^2)$ solves
\begin{align*}
 -\diverg(a\nabla \hat{w}) - \mu_C \hat{w}= -a\nabla w_C +  b w_C, \qquad \hat{w}\in H^1_0(\omega;\R^2),
\end{align*}
with its components orthogonal to $w_C$ in $L^2(\omega)$.
The second-order term $q_{\rm c}$ has the form 
\begin{align}\label{def:qc}
 q_{\rm c}(s)= B \xi(s) \cdot \xi(s) \qquad\text{with } B= \int_\omega a \bigl((\nabla\hat{w})^T + \nabla w_C\otimes x\bigr) w_C \dd{x}.
\end{align}
For $w_2$ one finds $w_2=\bar{w}\xi\cdot\xi$, where $\bar{w}\in H_0^1(\omega;\R^{2\times 2})$ is the solution of 
\begin{align*}
  -\diverg(a\nabla \bar{w}) - \mu_C \bar{w}= -a \bigl((\nabla \hat{w})^T + \nabla w_C\otimes x\bigr) 
+ b\otimes\hat{w}  + B w_C, \qquad w\in H^1_0(\omega),
\end{align*}
satisfying $\int_\omega \bar{w} w_C\dd{x}=0$. 

The asymptotic behavior of~\eqref{cross_sec_perturb} is made rigorous in the following analog of Proposition~\ref{prop:asymptotics_mu}.
\begin{proposition}\label{lem:asymptotics_mu_inhomo}
The following convergence holds uniformly in $[0,l]$:
\begin{align*}
\lim_{\delta\to 0} \frac{\mu_\delta(s) -\bigl(\mu_C+\delta(b\cdot \xi(s))\bigr)}{\delta^2} = q_{\rm c}(s),
\end{align*} 
with $q_{\rm c}$ defined in~\eqref{def:qc}.
Moreover, $w_\delta\to w_C$ in $H^1_0(Q_\length)$ as $\delta\to 0$, where $w_\delta(s,x)=w_\delta(s)(x)$ for $(s,x)\in Q_\length$.
\end{proposition}

\begin{proof} Since the arguments are similar, and even simpler, to the ones of Proposition~\ref{prop:asymptotics_mu},
we give only a sketch of the proof pointing out the differences. 

Plugging $\mu^\ast_\delta$ and $w^\ast_\delta$ of~\eqref{expansion_inhomogeneous} into \eqref{cross_sec_perturb},
we obtain (here again, we skip the explicit $s$-dependence) that
\begin{align*}
 \diverg(a \beta_\delta\nabla w^\ast_\delta)+ \mu^\ast_\delta\beta_\delta w^\ast_\delta& = 
\delta^3 \,\bigl[- \diverg(a(\xi\cdot x)w_2) + \mu_2 w_1 + \mu_1 w_2 \\
&\qquad\qquad\qquad-(\xi\cdot x)(\mu_1 w_1 + \mu_2 w_0 + \mu_0 w_2)
\bigr]\\
&+\  \delta^4 \,\bigl[\mu_2 w_2 - (\xi\cdot x) (\mu_2 w_1 + \mu_1 w_2)\bigr]\\
&+\  \delta^5\, \bigl[-(\xi\cdot x) \mu_2 w_2\bigr].
\end{align*}
The terms on the right-hand side are all bounded in $H^{-1}(\omega)$, so that we find
\begin{align*}
 \norm{ \diverg(a \beta_\delta\nabla w^\ast_\delta)+ \mu^\ast_\delta\beta_\delta w^\ast_\delta}_{H^{-1}(\omega)}\leq C\delta^3.
\end{align*}
Following the arguments of Lemma~\ref{lem:asymptotics_mueps} yields the estimates
\begin{align}\label{approximation_wdelta*}
\abs{\mu_\delta^\ast - \mu_\delta} \leq C\delta^3 \qquad\text{and}\qquad \norm{w_\delta^\ast - w_\delta}_{H^1_0(\omega)}  \leq C \delta^3
\end{align}
for $\delta>0$ sufficiently small.
Indeed, compared with~\eqref{est:weps_wasteps} and~\eqref{est:boundary}, we get an even better bound on $w_\delta^\ast-w_\delta$, since 
$w_\delta^\ast = 0$ on $\partial \omega$.
For fixed, small $\delta>0$, Lipschitz continuity of $s\mapsto \mu_\delta(s)$ and $s\mapsto w_\delta(s)$ with vanishing Lipschitz constants as $\delta\to 0$ 
can be shown following the proof of Lemma~\ref{lem:Lipschitz_continuity}. Hence, the proposition is proven.
\end{proof}

Depending on the behavior of the coefficient of order $\delta$ in the expansion 
of $\mu_\delta$, that is $(b\cdot \xi)$ (see~\eqref{def:b} for the definition of $b$), one has to distinguish between two different scenarios, 
namely propagation and localization.

\subsection{Propagation}\label{subsec:propagation}

The following theorem provides a sufficient condition for propagation through our tube-shaped waveguide. In the sequel, let
$(\eta_P^{(j)}, \varphi_P^{(j)})$ denote the eigenpair of order $j$ for
\begin{align*}
 -r \varphi'' + q \varphi= \eta \varphi, \qquad \varphi \in H_0^1(0,\length), 
\end{align*}
where $r=\int_\omega a w_C^2 \dd{x}$ and $q= q_\tau + q_{\rm c}$ with
\begin{align*}
q_\tau(s)=\tau(s)^2\int_\omega a\abs{\nabla w_C\cdot R x}^2\dd{x} -\tau'(s) \int_\omega a (\nabla w_C\cdot R x)w_C\dd{x}, \qquad s\in[0,\length],
\end{align*} 
and $q_{\rm c}$ defined in \eqref{def:qc}. Recall that $(\mu_C, w_C)$ is the first eigenpair of the cross section problem~\eqref{cross_sec}.

\begin{theorem}\label{theo:inhomo_propagation}
Let $b$ be given by~\eqref{def:b} and suppose that $(b\cdot \xi)$ is constant. For $j\in \N_0$, let $\{(\lambda_\delta^{(j)}, u_\delta^{(j)})\}_\delta$ 
be the sequence of $j$th eigenpairs for the spectral problem~\eqref{rescaledEVP}-\eqref{EVPepsilondelta} with $\eps=1$.
Then, for every $\delta>0$ one has
\begin{align*}
\lambda_\delta^{(j)}= \frac{\mu_C}{\delta^{2}} +\frac{(b\cdot \xi)}{\delta}+ \eta_{\delta}^{(j)},\qquad\text{where }\lim_{\delta\to 0}\eta^{(j)}_\delta=\eta_P^{(j)},
\end{align*}
and the sequence of eigenfunctions $\{u_\delta^{(j)}\}_\delta$ converges, up to a subsequence, in the following sense:
\begin{align*}
 u_\delta^{(j)}\circ \psi_{\delta}=v_\delta^{(j)}\quad\longrightarrow\quad 
v^{(j)}:=w_C \,\varphi^{(j)}_P \qquad \text{in $L^2(Q_\length)$ as $\delta \to 0$.}
\end{align*} 
Here, $\psi_{\delta}$ is the parameter transformation introduced in~\eqref{rescaling}. 

Conversely, any such $v^{(j)}$ is the $L^2(Q_\length)$-limit of a sequence $\{u_\delta^{(j)}\circ \psi_{\delta}\}_\delta$ with $u_\delta^{(j)}$ 
an eigenfunction of~\eqref{rescaledEVP}-\eqref{EVPepsilondelta} corresponding to the eigenvalue $\lambda_\delta^{(j)}$. 
\end{theorem}

\begin{proof}
With Proposition~\ref{lem:asymptotics_mu_inhomo} at hand the proof follows along the lines of that for Theorem~\ref{theo:mainresult_cylinder},
making use of Lemma~\ref{theo:Gammaconvergence}.
In particular, with $E_\delta, E_0:L^2(Q_\length)\to \overline\R$ given by
  \begin{align*}
    E_\delta [v]= \tilde{E}_{1,\delta}[v] - \int_{Q_\length}  \beta_\delta\frac{(\mu_C+ \delta (b\cdot\xi))}{\delta^2}\abs{v}^2\dd{s}\dd{x}
  \end{align*}
(see~\eqref{changedenergy} for the definition of $\tilde{E}_{1,\delta}$) and  
  \begin{align*}
    E_0[v]=\begin{cases}
    \displaystyle\int_0^\length r\, \abs{\varphi'}^2 +  q \abs{\varphi}^2\dd{s}, & 
    \text{if $v(s,x)=w_C(x) \varphi(s), \ (s,x) \in Q_\length,\ \varphi\in H_0^1(0,\length)$,}\\
    +\infty, & \text{otherwise},
    \end{cases}
  \end{align*}
it can be shown that $\Gamma$-$\displaystyle\lim_{\delta\to 0}E_\delta=E_0$ in the strong topology of $L^2(Q_\length)$.
\end{proof}

\begin{remark}
a) Notice that if $a\equiv 1$, we recover the results of~\cite{BMT07}. In this case, $b=\int_\omega a \nabla w_C w_C \dd{x}$ vanishes identically, 
so that the propagation condition is fulfilled 
with constant zero.

b) If $a$ is not constant, the propagation criterion is closely related to symmetry properties of the cross section. To see this, let $J\in O(2)$ such that $\rank(J-\Ibb)=2$.
If $J\omega=\omega$ and $a(Jx)=a(x)$ for $x\in \omega$, the symmetry relation carries over to $w_C$, and a change of variables gives $b=Jb$, which implies $b=0$. 

As an example, point symmetry of $a$ and $\omega$ regarding the origin corresponds to setting $J=-\Ibb$. 
Hence, for a waveguide with this type of symmetry the propagation condition is satisfied regardless of its curvature and torsion.
\end{remark}

\subsection{Localization of eigenmodes}\label{subsec:localization}
This section addresses the situation when the propagation condition is violated, meaning that $(b\cdot\xi)$ is not constant. 
We observe that localization of eigenmodes takes place around global minimizers of the function $h: s\mapsto (b\cdot \xi(s))$.
Whether these minimum points lie in the interior of $(0,\length)$ or are attained at the endpoints has an influence on the reduced limit problem.

For technical reasons we need to require higher regularity of the function $\xi$ in the following, 
precisely $\xi\in C^2([0,\length];\R^2)$ or $\xi\in C^1([0, \length];\R^2)$ as the case may be. 

\subsubsection{Localization at a single interior point}\label{subsubsec:inner_localization}
Let us assume first that  $s_0\in (0,\length)$ is a global minimizer for $h$. 

The methods used in this paragraph are inspired by those of~\cite{BMT_Robin}, where a localization result for homogeneous
waveguides with Robin boundary conditions is proven. Moreover, localization effects were observed for waveguides with 
varying cross section (see~\cite{FriedlanderSolomyak07, FriedlanderSolomyak09} for a study based on operator techniques).

Notice that in the asymptotic expansion~\eqref{expansion_inhomogeneous} of the cross section problem~\eqref{cross_sec_perturb} 
the $\delta$-order term of the eigenvalue  
$\mu_\delta$, which is $(b\cdot \xi(s))$, now depends on $s$. Thus, a reasoning as in Section~\ref{subsec:propagation} 
cannot be expected to give a satisfactory result. 
In fact,~\eqref{expansion_inhomogeneous} indicates that an appropriate refinement of the scales is needed. 

To this end, we use Taylor expansion to develop $h=(b\cdot \xi)$ around its unique global minimum point $s_0$.
Since $h'(s_0)=0$ and $h''(s_0)>0$, we get
\begin{align}\label{Taylor_expansion}
 h(s) =h(s_0) + \frac{1}{2} h''(\tilde{s})(s-s_0)^2, \qquad s\in [0,\length],  
\end{align}
with $\tilde{s}\in (s_0, s)$ if $s>s_0$ and $\tilde{s}\in (s,s_0)$ if $s<s_0$. 
Then, there exists a constant $K>0$ with
\begin{align}\label{upperlowerbound}
 K (s-s_0)^2\leq h(s) - h(s_0)\leq \frac{1}{K} (s-s_0)^2.
\end{align}

Notice that the energy functional 
\begin{align*}
E^\ast_\delta[v]=\tilde{E}_{1,\delta}[v]-\int_{Q_l} \beta_{\delta}\frac{(\mu_C+\delta h(s_0))}{\delta^2}\abs{v}^2\dd{s}\dd{x} \qquad
\end{align*} 
yields qualitative information regarding the concentration behavior of the system around $s_0$. 
Let $\{v_\delta\}_\delta\subset L^2(Q_l)$ be such that $\norm{v_\delta}_{L^2(Q_l)}\leq C$ and $E^\ast_\delta[v_\delta]\leq C$ for all $\delta>0$ with a constant $C>0$. 
In view of~\eqref{upperlowerbound} and the lower bound 
\begin{align*}
 \int_{Q_l} \frac{\mu_\delta-(\mu_C+\delta h)}{\delta^2}\abs{v}^2\dd{s}\dd{x}\geq -c\norm{v}_{L^2(Q_l)}^2,
\end{align*} 
which in turn is a consequence of Proposition~\ref{lem:asymptotics_mu_inhomo} in combination with~\eqref{upperlowerbound}, we find that
\begin{align*}
C\delta^2&\geq \delta^2 E^\ast_\delta[v_\delta]\geq  \int_{Q_l} a\beta_\delta \abs{\nabla v_\delta}^2 -\beta_\delta(\mu_C+\delta h)\abs{v_\delta}^2
+ \delta \beta_\delta K (s-s_0)^2\abs{v_\delta}^2 \dd{s}\dd{x} \\ &\geq -c\delta^2 \norm{v_\delta}_{L^2(Q_l)}^2 
+   \delta K \int_{Q_l} \beta_\delta (s-s_0)^2\abs{v_\delta}^2 \dd{s}\dd{x}.
\end{align*}
Using the uniform bound on the $L^2$-norm of $\{v_\delta\}_\delta$, one obtains the estimate
\begin{align*}
\int_{Q_l}\beta_\delta (s-s_0)^2  \abs{v_\delta}^2\dd{s}\dd{x}\leq c\delta.
\end{align*}
Due to the uniform convergence of $\beta_\delta$ to $1$ this leads to
\begin{align*}
\int_{Q_l} (s-s_0)^2 v_\delta^2\dd{s}\dd{x}\to 0
\end{align*}
as $\delta\to 0$, which implies that $\{v_\delta^2\}_\delta$ converges weakly, up to a subsequence, to a measure supported in $\{s_0\}\times\omega$.
Considering a sequence $\{v_\delta\}_\delta$ as above, but with $L^2$-norm bounded away from zero, for instance $\norm{v_\delta}_{L^2(Q_l)}=1$, 
this shows that a concentration effect 
around $s_0$ is produced. 
To obtain maximal quantitative information about the behavior of the system locally around $s_0$, we have to use a blow-up 
argument and an appropriately rescaled version of $\{E^\ast_\delta\}_\delta$.

Following~\cite[Remark~4.3]{BMT07}, we perform a change of variables by replacing $s\in [0,\length]$ with
\begin{align}\label{rescalings1}
 t=\delta^{-1/4}(s-s_0), \qquad t\in I_\delta:=\delta^{-1/4}[-s_0, \length-s_0],
\end{align}
and we set
\begin{align}\label{rescalingz}
 z(t,x)=\delta^{1/8}v(\delta^{1/4}t + s_0, x)=\delta^{1/8}v(s,x), \qquad t\in I_\delta,\ s\in [0,\length],\ x\in \omega.
\end{align}
This definition of $z$ preserves the $L^2$-norm, i.e.\ $\norm{z}_{L^2(I_\delta\times \omega)}= \norm{v}_{L^2(Q_\length)}$.
In our notation, the quantities transformed according to~\eqref{rescalings1} are marked with bars. 
For instance, we use $\bar{\beta}_\delta(t)=\beta_\delta(s)$, $\bar{\xi}(t)=\xi(s)$, $\bar{\theta}(t)=\theta(s)$ and 
$\bar{h}(t)=h(s)$.
Notice that due to~\eqref{rescalings1}, also $\bar{\xi}$, $\bar{\theta}$ and $\bar{h}$ depend on $\delta$.

Then the rescaled cross section problem for fixed $t\in I_\delta$ reads
\begin{align}\label{cross_sec_perturb_loca}
 -\diverg(a \bar{\beta}_\delta(t)\nabla \bar{w})=\bar{\mu}\bar{\beta}_\delta(t) \bar{w}, \qquad \bar{w}\in H^1_0(\omega).  
\end{align}
For the first eigenpair $(\bar{\mu}_\delta(t), \bar{w}_\delta(t))$, one has $\bar{\mu}_\delta(t)=\mu_\delta(s)$ and $\bar{w}_\delta(t)=w_\delta(s)$ 
for $t$ and $s$ selected as in \eqref{rescalings1} recalling that $(\mu_\delta(s), w_\delta(s))$ is the first eigenpair of~\eqref{cross_sec_perturb}.
As in Sections~\ref{subsec:asymptotics} and~\ref{subsec:asymptotics_inhomogeneous} a detailed analysis of the cross section problem, 
in particular of its behavior as $\delta$ tends to $0$ is essential.

\begin{lemma}\label{lem:asymptotics_mu_loc1}
There exists a constant $c>0$ such that 
\begin{align}\label{lowerbound}
 \frac{\bar{\mu}_\delta(t)-(\mu_C +\delta h(s_0))}{\delta^{3/2}}\geq Kt^2-c\delta^{1/2}
\end{align}
for all $\delta>0$ and $t\in I_\delta$.
For any $I\subset \R$ compact, the following convergence holds uniformly in $I$:
\begin{align*}
   \lim_{\delta\to 0}\frac{\bar{\mu}_\delta(t)-(\mu_C +\delta h(s_0))}{\delta^{3/2}}= q(t),
\end{align*} 
with $q(t):=\frac{1}{2}h''(s_0)t^2$ for $t\in \R$. Moreover, $\bar{w}_\delta \to w_C$ in $H^1(I \times \omega)$ as $\delta\to 0$, where $\bar{w}(t,x)=\bar{w}(t)(x)$ for
$(t,x)\in I\times \omega$.
\end{lemma}

\begin{proof} 
The estimate \eqref{lowerbound} for $\delta>0$ and $t\in I_\delta$ follows from
\begin{align*}
 \bar{\mu}_\delta(t)-(\mu_C+\delta h(s_0)) &= \bigl(\mu_\delta(s)-\mu_C-\delta h(s)\bigr) + \delta (h(s)-h(s_0)) \\
&\geq -c\delta^2 
+ \delta K(s-s_0)^2=-c\delta^2 + \delta^{3/2} K t^2,
\end{align*}
where we used \eqref{approximation_wdelta*}, \eqref{upperlowerbound} and \eqref{rescalings1}.
Now let $t\in I$. Then, by \eqref{Taylor_expansion}, Proposition~\ref{lem:asymptotics_mu_inhomo} and the regularity of $\xi$, 
\begin{align*}
\absB{\frac{\bar{\mu}_\delta(t)-(\mu_C +\delta h(s_0))}{\delta^{3/2}}- q(t)} 
& \leq \absB{\frac{\mu_\delta(s)-(\mu_C+\delta h(s))}{\delta^{3/2}}} + \frac{1}{2}\absB{h''(\tilde{s})\frac{(s-s_0)^2}{\delta^{1/2}}-h''(s_0)t^2}\\
&\leq C \delta^{1/2} + \frac{t^2}{2} \abs{h''(\tilde{s})-h''(s_0)}
\end{align*}
with $s=\delta^{1/4}t+s_0$ and $\tilde{s}$ between $s$ and $s_0$.
The uniform continuity of $h''$ on $[0,\length]$ together with $\abs{\tilde{s}-s_0}\leq\delta^{1/4}t\leq C\delta^{1/4}$ implies 
the second part of the assertion.

To prove convergence of the first eigenfunctions of~\eqref{cross_sec_perturb_loca}, we use 
\begin{align*}
 \lim_{\delta\to 0}\norm{\bar{w}_\delta(t)- w_C}_{H^1_0(\omega)} =0,
\end{align*}
which holds uniformly in $I$ and is an immediate consequence 
of Proposition~\ref{lem:asymptotics_mu_inhomo}, in conjunction with the Lipschitz continuity of the map $t\mapsto \bar{w}_\delta(t)$ with 
Lipschitz constant $C\delta^{3/4}$. The latter follows from Lemma~\ref{lem:Lipschitz_continuity} accounting for~\eqref{rescalings1}. 
\end{proof}

In the following, we identify $H_0^1( I_\delta\times\omega)$-functions with 
elements in $H^1(\R\times\omega)$ through a trivial extension by zero, and we define for an open set $U\subset\R$ and $p\geq 1$,
\begin{align}\label{def:L2t}
 L^2(U;t^p\dd{t}):=\Bigl\{\phi:U\to \R \text{ measurable}\,:\, \int_\R \abs{t}^p \abs{\phi(t)}^2\dd{t}<\infty\Bigr\}.
\end{align}
The next theorem gives the spectral asymptotics of \eqref{rescaledEVP}-\eqref{EVPepsilondelta} within the setting of this paragraph, showing effective harmonic 
oscillations in the neighborhood of the unique minimum point of $h$.
Let ($\eta_{L}^{(j)}, \phi_{L}^{(j)})$ be the $j$th eigenpair of 
the one-dimensional harmonic oscillator defined by
\begin{align}\label{harmonicoscillator}
 -r\ddot{\phi} + q \phi= \eta \phi, \qquad\text{$\phi \in H^1(\R)\cap L^2(\R;t^2\dd{t})$,}
\end{align}
which is to be understood in the weak sense (see~\cite[Chapter~VIII, Section~7.4]{DautrayLions90}).
Here, $r=\int_\omega a w_C^2\dd{x}$ and $q(t)=\frac{1}{2}h''(s_0)t^2$ for $t\in \R$. Note that $\dot{(\cdot)}= \frac{d}{dt}$.
\begin{theorem}\label{theo:localization}
Suppose that $\xi\in C^2([0,\length];\R^2)$ and $h$ attains its global minimum at the unique point $s_0\in (0,\length)$. 
For $j\in \N_0$, let $\{(\lambda_\delta^{(j)}, u_\delta^{(j)})\}_\delta$ 
be a sequence of $j$th eigenpairs for the spectral problem~\eqref{rescaledEVP}-\eqref{EVPepsilondelta} with $\eps=1$.
Then, for every $\delta>0$ one has
\begin{align*}
\lambda_\delta^{(j)}= \frac{\mu_C}{\delta^{2}} +\frac{h(s_0)}{\delta}+ \frac{\eta_{\delta}^{(j)}}{\delta^{1/2}},
\qquad\text{where }\lim_{\delta\to 0}\eta^{(j)}_\delta=\eta_L^{(j)},
\end{align*}
and the sequence of eigenfunctions $\{u_\delta^{(j)}\}_\delta$ converges, up to a subsequence, in the following sense:
\begin{align*}
 \delta^{1/8}u_\delta^{(j)}\circ \psi_{\delta}\circ \varsigma_{\delta, s_0}=z_\delta^{(j)}\quad\longrightarrow\quad 
z^{(j)}:=w_C \,\phi^{(j)}_L \qquad \text{in $L^2(\R\times\omega)$ as $\delta \to 0$.}
\end{align*} 
Here, $\psi_{\delta}$ is defined in~\eqref{rescaling} and 
$\varsigma_{\delta, s_0}:I_\delta\times \omega \to Q_\length$ is the change of variables defined by 
$\varsigma_{\delta, s_0}(t,x)=(\delta^{1/4}t+s_0,x)$.

Conversely, any such $z^{(j)}$ is the $L^2(\R\times \omega)$-limit of a sequence 
$\{\delta^{1/8}u_\delta^{(j)}\circ \psi_{\delta}\circ\varsigma_{\delta, s_0}\}_\delta$, where
$u_\delta^{(j)}$ is
an eigenfunction of~\eqref{rescaledEVP}-\eqref{EVPepsilondelta} with $\eps=1$ corresponding to the eigenvalue $\lambda_\delta^{(j)}$. 
\end{theorem}
\begin{proof}
Consider for $\delta>0$ the functionals $E_\delta: L^2(Q_\length)\to \overline{\R}$ given by 
\begin{align}\label{Edeltav}
 E_\delta[v] = \tilde{E}_{1,\delta}[v] - 
\int_{Q_\length} \beta_\delta\frac{(\mu_C+ \delta h(s_0))}{\delta^2} \abs{v}^2 \dd{s}\dd{x}
\end{align}
for $v\in H_0^1(Q_\length)$ and $E_\delta[v]= + \infty$, otherwise.
Using the change of variables~\eqref{rescalings1} and~\eqref{rescalingz}, we may define a new family of rescaled functionals $\{\bar{E}_\delta\}_\delta$ by setting
\begin{align*}
 \bar{E}_\delta[z]=\delta^{1/2}E_\delta[v]
\end{align*} 
for $\delta>0$. 
Precisely, this yields that
$\bar{E}_\delta:L^2(\R\times\omega)\to \overline{\R}$ is finite for $z\in H_0^1(I_\delta\times\omega)$ with 
\begin{align*}
 \bar{E}_\delta[z] =  \int_{I_\delta\times \omega} \frac{a}{\bar{\beta}_\delta} \abs{\dot{z} +\delta^{1/4}\bar{\tau}(\nabla^x z\cdot Rx)}^2+ 
\frac{\bar{\beta}_\delta}{\delta^{3/2}} \bigl(a\abs{\nabla^x z}^2 - (\mu_C+ \delta h(s_0))\abs{z}^2\bigr)\dd{t}\dd{x}.
\end{align*}
The motivation for choosing $\delta^{1/2}$ as the scaling factor for
$E_\delta$ results from the need of balancing terms in order to obtain control over the cross section part
of $\bar{E}_\delta$ (see Lemma~\ref{lem:asymptotics_mu_loc1}).
Let $\bar{E}_0: L^2(\R\times\omega)\to \overline{\R}$ be given through
\begin{align*}
\bar{E}_0[z]=\begin{cases}
\displaystyle  \int_\R r\abs{\dot{\phi}}^2 +  q \abs{\phi}^2\dd{t}, &
 \text{if $z(t,x)=w_C(x) \phi(t),\ (t,x)\in \R\times \omega,$} \\ & \qquad\phi\in H^1(\R)\cap L^2(\R;t^2\dd{t}),\\
+\infty, & \text{otherwise}.
 \end{cases}
\end{align*}
The proof will proceed in two steps. 
In the first step we establish that $\{\bar{E}_\delta\}_\delta$ satisfies the 
conditions of Lemma~\ref{theo:Gammaconvergence}, while in the second we apply the latter to obtain the stated asymptotic behavior of the spectrum.

\textit{Step~1: Verifying conditions $(i)$-$(iii)$ of Lemma~\ref{theo:Gammaconvergence} for $\{\bar{E}_\delta\}_\delta$.} 
The proof of the $\Gamma$-convergence result 
\begin{align}\label{Gammalimit_Ebar}
 \Gamma\text{-}\lim_{\delta \to 0}\bar{E}_\delta=\bar{E}_0
\end{align} with respect to the strong topology of $L^2(\R\times \omega)$
is mostly analogous to that of Proposition~\ref{prop:Gammaconvergence_Heps}, which implies $(iii)$ and, at the same time, also $(i)$.
When proving compactness of bounded energy sequences, i.e.~$(ii)$, however, a finer argument inspired by~\cite[Proof of Proposition~4.4]{BMT_Robin} is necessary. 

Let $\{z_\delta\}_\delta$ be a bounded sequence in $L^2(\R\times\omega)$ such that 
$\sup_{\delta>0} {\bar{E}_\delta[z_\delta]}\leq C<\infty$. Then, $z_{\delta}\in H_0^1( I_\delta\times\omega)$ and 
$\norm{z_\delta}_{H^1(\R\times\omega)}\leq C$ for all $\delta>0$, so that we can extract a subsequence of $\{z_\delta\}_\delta$ satisfying
\begin{align}\label{weakconvergencezdelta}
 z_\delta\weakly z\qquad \text{in $H^1(\R\times\omega)$}
\end{align}
for some $z\in H^1(\R\times\omega)$ as $\delta\to 0$.
Since the Rellich-Krondrachov theorem is not valid in unbounded domains, in order to derive strong 
$L^2$-convergence of $\{z_\delta\}_\delta$ 
we argue as follows. 

The lower bound \eqref{lowerbound} of Lemma~\ref{lem:asymptotics_mu_loc1} implies that
\begin{align*}
 \bar{E}_\delta[z]\geq \int_{I_\delta\times \omega} \bar{\beta}_\delta(Kt^2-c\delta^{1/2})\abs{z}^2\dd{t}\dd{x}
\end{align*}
for any $z\in H_0^1(I_\delta\times \omega)$, so that, for all $\delta>0$,
\begin{align}\label{unifrombound_loca}
\int_{\R\times\omega} t^2 \abs{z_\delta}^2\dd{t}\dd{x}\leq C
\end{align}
with $C>0$ independent of $\delta$.
For given $\eps>0$, let $\rho>0$ be such that $C/\rho^2< \eps/5$ with the constant $C$ as in \eqref{unifrombound_loca} and 
$\int_{ \{\abs{t}\geq \rho\}\times\omega}\abs{z}^2 \dd{t}\dd{x}< \eps/5$.
Then, for $\delta$ sufficiently small,
\begin{align*}
 \norm{z_\delta-z}_{L^2(\R\times\omega)}^2\leq \int_{ \{\abs{t}< \rho\}\times\omega} \abs{z_\delta-z}^2 \dd{t}\dd{x} 
+ 2\int_{\{\abs{t}\geq \rho\}\times \omega} \abs{z_\delta}^2 \dd{t} \dd{x}+ 2 \eps/5 < \eps. 
\end{align*}
Indeed, in view of~\eqref{weakconvergencezdelta}, applying the Rellich-Krondrachov theorem to the first term 
yields the existence of a $\delta_0>0$ such that 
$\int_{ \{\abs{t}< \rho\}\times\omega} \abs{z_\delta-z}^2 \dd{t}\dd{x} <\eps/5$ for all $\delta<\delta_0$. For the second term we used
\eqref{unifrombound_loca} to derive
\begin{align*}
 \int_{\{\abs{t}\geq \rho\}\times\omega}\abs{z_\delta}^2 \dd{t}\dd{x}\leq \frac{1}{\rho^2} \int_{ \{\abs{t}\geq \rho\}\times\omega}t^2 \abs{z_\delta}^2 \dd{t}\dd{x}
\leq \frac{C}{\rho^2}<\eps/5.
\end{align*}
Since $\eps>0$ was arbitrary, $z_\delta \to z$ in $L^2(\R\times\omega)$ for $\delta\to 0$.
Besides, the lower semicontinuity of the $L^2$-norm in conjunction with~\eqref{unifrombound_loca} entails that $\norm{z}_{L^2(\omega)}\in L^2(\R;t^2\dd{t})$.
This concludes the proof of compactness for bounded energy sequences.

\textit{Step~2: Applying Lemma~\ref{theo:Gammaconvergence} to $\{\bar{E}_\delta\}_\delta$.}
For this purpose, we define operators
$\bar{\Bcal}_\delta: L^2(\R\times\omega)\to L^2(\R\times\omega)$ with $\Dcal(\bar{\Bcal}_\delta)= H_0^1(I_\delta\times\omega)$ such that
\begin{align*}
 \bar{E}_\delta[z]=(\bar{\Bcal}_\delta z, z)_{L^2(I_\delta\times\omega)}, \qquad z\in \Dcal(\bar{\Bcal}_\delta).
\end{align*}
Then, as $\delta$ tends to zero, Lemma~\ref{theo:Gammaconvergence} implies convergence of eigenpairs of $\bar{\Bcal}_\delta$ to those
of the limit operator $\bar{\Bcal}_0$, which, in view of~\eqref{Gammalimit_Ebar}, reads 
$\bar{\Bcal}_0:L^2(\R\times \omega)\to L^2(\R\times \omega)$,
\begin{align*}
\bar{\Bcal}_0 z= (-r \ddot{\phi}+ q \phi)w_C, \qquad \text{$z\in \Dcal(\bar{\Bcal}_0)$,}   
\end{align*}
with $\Dcal(\bar{\Bcal}_0)=\{z\in L^2(\R\times \omega): z=w_C\phi,\, \phi\in H^1(\R)\cap L^2(\R;t^2\dd{t})\}$.
The proof is concluded by the observation that $(\sigma_\delta, z_\delta)$ is an eigenpair for $\bar{\Bcal}_\delta$ if and only if 
$(\delta^{-1/2}\sigma_\delta, v_\delta)$
is an eigenpair for $\Bcal_\delta:=\Acal_{1,\delta} - \mu_C\delta^{-2} - h(s_0)\delta^{-1}$
(see~\eqref{EVPepsilondelta} for the definition of $\Acal_{1, \delta}$), and $z_\delta$ is related with $v_\delta$ through~\eqref{rescalingz}, i.e.~$z_\delta=\delta^{1/8} v_\delta\circ \varsigma_{\delta, s_0}$.
\end{proof}

\subsubsection{Localization at multiple interior points}\label{subsubsec:multiple_localization}
Next we deal with the case of $h=(b\cdot \xi)$ having a global minimum in $(0,\length)$ attained at exactly 
two distinct points in $(0,l)$. 
The results of this paragraph can be easily generalized to the situation
where $h$ has an arbitrary finite number of global minimizers.

Let $s_1, s_2\in (0,l)$ with $s_1<s_2$ be two minimum points of $h$ such that $h(s_1)=h(s_2)=:h_0$, $h'(s_1)=h'(s_2)=0$ and $h''(s_1), h''(s_2)>0$.
In view of the smooth two-well structure of $h$ there is a constant $K>0$ such that
\begin{align*}
 K\min\{(s-s_1)^2, (s-s_2)^2\}\leq h(s)-h_0\leq \frac{1}{K} \min\{(s-s_1)^2, (s-s_2)^2\}
\end{align*}
for all $s\in[0,\length]$.
Using Taylor expansion around the points $s_1$ and $s_2$ one finds the following two representations of $h$,
\begin{align*}
 h(s) =h_0 + \frac{1}{2} h''(\tilde{s}_1)(s-s_1)^2 \qquad\text{and}\qquad
 h(s) =h_0 + \frac{1}{2} h''(\tilde{s}_2)(s-s_2)^2,\qquad s\in [0,\length],  
\end{align*}
with $\tilde{s}_i\in (s_i, s)$ if $s>s_i$ and $\tilde{s}_i\in (s,s_i)$ if $s<s_i$, $i=1,2$. 

In this case we have to deal with concentration in both points $s_1$ and $s_2$ and localize around each of them.
Considering the family of rescaled functionals $\{\delta^{1/2}E_\delta\}_\delta$ (see~\eqref{Edeltav} for the definition of $E_\delta$), 
the strategy is to split $[0,l]$ into two intervals
$J^1=[0, \hat{s}]$ and $J^2=[\hat{s}, \length]$ with $s_1<\hat{s}<s_2$, and to proceed with appropriate changes of variables in each of them, namely 
\begin{align*}
t &= \delta^{-1/4}(s-s_1), \qquad  t\in I_\delta^1=:[l^1_\delta, \hat{l}_\delta^{\,1}], \ s\in J^1,\\
t &= \delta^{-1/4}(s-s_2), \qquad  t\in I_\delta^2=:[\hat{l}^{\,2}_\delta, l_\delta^{2}], \ s\in J^2.
 \end{align*}
For $i=1,2$ and $v\in H_0^1(Q_\length)$, let 
\begin{align}\label{rescaling_doublewell}
 z_i(t,x)=\delta^{1/8} v(\delta^{1/4}t + s_i,x)=\delta^{1/8}v(s,x), \qquad t\in I_\delta^i,\ s\in J^i, \ x\in \omega.
\end{align}
Then, $z=(z_1, z_2)$ lies in the space
\begin{align*}
\Zcal_\delta &=\bigl\{z:=(z_1,z_2) \in H^1(I_\delta^1\times \omega)\times H^1(I_\delta^2\times \omega)\ : \ z_i=0 
\text{ on $\partial\bigl((I_\delta^i\setminus\{\hat{l}_\delta^{\,i}\})\times \omega\bigr)$,} \\
&\qquad\qquad\qquad\qquad\qquad\qquad\qquad\qquad\qquad\qquad\qquad\qquad \ 
z_1|_{\{\hat{l}_\delta^{\, 1}\}\times \omega} = z_2|_{\{\hat{l}_\delta^{\,2}\}\times \omega}\bigr\}.
\end{align*}
Notice that the mapping $H_0^1(Q_\length) \to \Zcal_\delta$, $v\mapsto z$ induced by \eqref{rescaling_doublewell}
is a bijection, and that $\norm{z}_{L^2(I_\delta^1\times \omega)\times L^2(I_\delta^2\times \omega)}= \norm{v}_{L^2(Q_\length)}$.

In analogy to Section~\ref{subsubsec:inner_localization} we denote the quantities transformed according to~\eqref{rescaling_doublewell} with bars. 
Besides, we add subscript indices $"1"$ or $"2"$ to 
specify the change of variables used. 

The transformed energy
$\bar{E}_\delta:[L^2(\R\times \omega)]^2\to \overline{\R}$ takes the form
\begin{align*}
 \bar{E}_\delta[z] &= \bar{E}^1_\delta[z_1] + \bar{E}^2_\delta[z_2],
\end{align*}
where $\bar{E}_\delta^i:L^2(\R\times\omega)\to \overline{\R}$ for $i=1,2$ is finite if $z_i\in H^1(I_\delta^i\times \omega)$
with
\begin{align*}
 \bar{E}_\delta^i[z_i]=\int_{I_\delta^i\times\omega} \frac{a}{(\bar{\beta}_\delta)_i} \abs{\dot{z}_i 
+\delta^{1/4}\bar{\tau}_i(\nabla^x z_i\cdot Rx)}^2+ 
\frac{(\bar{\beta}_\delta)_i}{\delta^{3/2}} \left(a\abs{\nabla^x z_i}^2 - (\mu_C+ \delta h_0)\abs{z_i}^2 \right)\dd{t}\dd{x}.
\end{align*}

Let us introduce the limit functional
$\bar{E}_0: [L^2(\omega\times\R)]^2\to \overline{\R}$ given by
\begin{align*}
\bar{E}_0[z]=\begin{cases}
\displaystyle\sum_{i=1}^2\int_\R r\, \abs{\dot{\phi_i}}^2 +  q_i \abs{\phi_i}^2\dd{t}, & 
  \text{if $z(t,x)=w_C(x)\phi(t), \ (t,x) \in \R\times\omega,$} \\  &  \qquad \phi\in [H^1(\R)\cap L^2(\R;t^2\dd{t})]^2,\\[0.2cm]
 +\infty, & \text{otherwise,}
\end{cases}
\end{align*}
and let ($\eta_{L}^{(j)}, \phi_{L}^{(j)})$ be the $j$th eigenpair of the system of two one-dimensional harmonic oscillators 
\begin{align}\label{twooscillators}
 -r\ddot{\phi} + Q\phi= \eta \phi, \qquad\text{ $\phi\in [H^1(\R)\cap L^2(\R;t^2\dd{t})]^2$}
\end{align}
with $r=\int_\omega a w_C^2\dd{x}$ and $Q(t) = q_1(t)e_1\otimes e_1 + q_2(t)e_2\otimes e_2 \in\R^{2\times 2}$ for $t\in \R$, where 
$q_i(t)=\frac{1}{2}h''(s_i)t^2$, $i=1,2$.

Then, the asymptotic spectral analysis of~\eqref{rescaledEVP}-\eqref{EVPepsilondelta} with $\eps=1$ 
gives the following. 
\begin{theorem}
Suppose that $\xi\in C^2([0,\length];\R^2)$ and $h$ attains its global minimum $h_0$ 
at the interior points $s_1$ and $s_2$ with $s_1<s_2$.
For $j\in \N_0$, let $\{(\lambda_\delta^{(j)}, u_\delta^{(j)})\}_\delta$ 
be a sequence of $j$th eigenpairs for the spectral problem~\eqref{rescaledEVP}-\eqref{EVPepsilondelta} with $\eps=1$.
Then, for every $\delta>0$ one has
\begin{align*}
\lambda_\delta^{(j)}= \frac{\mu_C}{\delta^{2}} +\frac{h_0}{\delta}+ \frac{\eta_{\delta}^{(j)}}{\delta^{1/2}},
\qquad\text{where }\lim_{\delta\to 0}\eta^{(j)}_\delta=\eta_L^{(j)},
\end{align*}
and the sequence of eigenfunctions $\{u_\delta^{(j)}\}_\delta$ converges, up to a subsequence, in the following sense:
\begin{align*}
 \delta^{1/8}u_\delta^{(j)}\circ \psi_{\delta}\circ \varsigma_{\delta, s_i}=(z_\delta^{(j)})_i\quad\longrightarrow\quad 
z^{(j)}_i:=w_C \,(\phi^{(j)}_L)_i \quad \text{in $L^2(\R\times\omega)$ as $\delta \to 0$}
\end{align*} for $i=1,2$.
Here, $\psi_{\delta}$ is defined in~\eqref{rescaling}, and 
$\varsigma_{\delta, s_i}:I_\delta^i\times \omega \to J^i$ represents the changes of variables 
$\varsigma_{\delta, s_i}(t,x)=(\delta^{1/4}t+s_i,x)$ for $(t,x)\in I_\delta^i\times \omega$, where $J^1=[0, \hat{s}]$ and
$J^2=[\hat{s}, \length]$ with $\hat{s}=(s_2-s_1)/2$, and $I_\delta^i=\delta^{-1/4}(J^i-s_i)$.

Conversely, any such $z^{(j)}_i$ is the $L^2(\R\times \omega)$-limit of a sequence 
$\{\delta^{1/8}u_\delta^{(j)}\circ \psi_{\delta}\circ\varsigma_{\delta, s_i}\}_\delta$, $i=1,2$,
where $u_\delta^{(j)}$ is an eigenfunction of~\eqref{rescaledEVP}-\eqref{EVPepsilondelta} with $\eps=1$ corresponding to the eigenvalue $\lambda_\delta^{(j)}$. 
\end{theorem}
\begin{remark}
Note that the spectrum $\{\eta_{L}^{(j)}\}_j$ of~\eqref{twooscillators} is the union of the
spectra of~\eqref{harmonicoscillator} in Section~\ref{subsubsec:inner_localization} with $s_0=s_1$ and $s_0=s_2$, respectively.  
Since the eigenfunctions of~\eqref{twooscillators} have contributions in both components,
oscillations around both minimum points of $h$ are observed, in general with different intensity and phase.
\end{remark}

\begin{proof}
The idea of the proof is to perform the arguments of the previous paragraph on localization at single interior points separately
for $\bar{E}_\delta^i$ with $s_0=s_i$, $i=1,2$, and merge them afterwards. The only difficulty is to mind the coupling non-zero boundary conditions
on $\{\hat{l}_\delta^i\}\times \omega$. Since the boundary moves to infinity as $\delta$ becomes small, this, however, does not affect 
the final result, meaning that one observes a splitting into two decoupled oscillators in the limit. 
 
As previously, we first prove that $\{\bar{E}_\delta\}_\delta$ satisfies the conditions of Lemma~\ref{theo:Gammaconvergence} and 
then apply it to derive the stated spectral behavior.
In analogy to the proof of Theorem~\ref{theo:localization}, one can then apply Lemma~\ref{theo:Gammaconvergence} to $\{\bar{E}_\delta\}_\delta$, which
implies the assertion.
\end{proof}

\subsubsection{Localization at the endpoints} 
In this paragraph, we consider the case where the global minimum of $h:[0,\length]\to \R$ is attained at one of the endpoints $s_0\in \{0, l\}$. 
Without loss of generality we choose $s_0=0$. Note that in the spirit of Section~\ref{subsubsec:multiple_localization} 
for multiple interior localization points, one can as well treat the case, 
where both endpoints $0$ and $\length$ are global minimizers of $h$.

Let us denote by $(\eta_{L}^{(j)}, \phi_{L}^{(j)})$ the $j$th eigenpair of 
the one-dimensional eigenvalue problem
\begin{align*}
 -r\ddot{\phi} + q \phi= \eta \phi, \qquad\text{$\phi \in H^1(\R^+)\cap L^2(\R^+;t\dd{t})$,}
\end{align*}
where $r=\int_\omega a w_C^2\dd{x}$, $q(t)=h'(0)t$ for $t\in \R^+$ and $L^2(\R^+;t\dd{t})$ is as defined in \eqref{def:L2t}.
With this definition the asymptotic behavior of the eigenvalue problem~\eqref{rescaledEVP}-\eqref{EVPepsilondelta} with $\eps=1$ can be 
expressed in the following theorem.
\begin{theorem}\label{theo:localization_multipleenpoints}
Suppose that $\xi\in C^1([0,\length];\R^2)$ and $h$ has a global minimum attained only at $0$ with $h'(0)\neq 0$. 
Here, $h'(0)$ denotes the right-sided derivative of $h$ at $0$.
For $j\in \N_0$, let $\{(\lambda_\delta^{(j)}, u_\delta^{(j)})\}_\delta$ 
be a sequence of $j$th eigenpairs for the spectral problem~\eqref{rescaledEVP}-\eqref{EVPepsilondelta} with $\eps=1$.
Then, for every $\delta>0$ one has
\begin{align*}
\lambda_\delta^{(j)}= \frac{\mu_C}{\delta^{2}} +\frac{h(0)}{\delta}+ \frac{\eta_{\delta}^{(j)}}{\delta^{2/3}},
\qquad\text{where }\lim_{\delta\to 0}\eta^{(j)}_\delta=\eta_L^{(j)},
\end{align*}
and the sequence of eigenfunctions $\{u_\delta^{(j)}\}_\delta$ converges, up to a subsequence, in the following sense:
\begin{align*}
 \delta^{1/6}u_\delta^{(j)}\circ \psi_{\delta}\circ \varsigma_{\delta}=z_\delta^{(j)}\quad\longrightarrow\quad 
z^{(j)}:=w_C \,\phi^{(j)}_L \qquad \text{in $L^2(\R^+\times\omega)$ as $\delta \to 0$.}
\end{align*} 
Here, $\psi_{\delta}$ is defined in~\eqref{rescaling}, and 
$\varsigma_{\delta}: [0,\delta^{-1/3}\length]\times \omega \to Q_\length$ is defined by 
$\varsigma_{\delta}(t,x)=(\delta^{1/3}t,x)$.

Conversely, any such $z^{(j)}$ is the $L^2(\R^+\times \omega)$-limit of a sequence 
$\{\delta^{1/6}u_\delta^{(j)}\circ \psi_{\delta}\circ\varsigma_{\delta}\}_\delta$, where $u_\delta^{(j)}$ is 
an eigenfunction of~\eqref{rescaledEVP}-\eqref{EVPepsilondelta} with $\eps=1$ corresponding to the eigenvalue $\lambda_\delta^{(j)}$. 
\end{theorem}

\begin{proof}
Developing $h$ with $h'(0)>0$ around the left endpoint $0$ by Taylor expansion, one obtains
\begin{align*}
 h(s) =h(0) + h'(\tilde{s})s, \qquad s\in (0,\length],  
\end{align*}
with $\tilde{s}\in (0, s)$. There is a constant $K>0$ such that for all $s\in [0,l]$,
\begin{align*}
 Ks\leq \ h(s)-h(0)\ \leq \frac{1}{K}s.
\end{align*}
We perform a change of variables replacing $s\in [0,\length]$ by
\begin{align*}
 t=\delta^{-1/3}s, \qquad t\in I_\delta:= [0,\delta^{-1/3}\length],
\end{align*}
and let 
\begin{align*}
 z(t,x)= \delta^{1/6}v(s,x), \qquad t\in I_\delta, \ s\in [0,\length],\ x\in \omega.
\end{align*} 
From this point on, one may proceed as in the proof of Theorem~\ref{theo:localization} in Section~\ref{subsubsec:inner_localization}.
\end{proof}

\appendix
\section{Explicit formulas for the fourth-order coefficients}\label{appendix_A}
\subsection{The torsion-free case}\label{appendix_A1} For the sake of completeness, we give here the precise definition 
of the fourth-order coefficient $w_4$ in the development $w_\eps^\ast$ of $w_\eps$ in~\eqref{asymptotic_expansion}. 
By formal computation one obtains that for $s\in[0,l]$, $x\in \omega$, and the fast variable $y=x/\eps$,
\begin{align}\label{w4}
w_4(s,x,y)&=\Pi_{ijkl}(y)\partial^x_{ijkl}w_H(x) + \Lambda_{ijk}(y)\partial_{ijk}^x\bar{w}_1(s,x) 
+ \zeta_{ij}(y)\partial_{ij}^x\bar{w}_2(s,x)\nonumber \\
&\qquad+ \nu_{ijk}(y)\xi_k(s)\partial_{ij}^x w_H(x) + \varkappa_{ij}(y)\xi_j(s)(\xi(s)\cdot x)\partial_i^x w_H(x) \\
&\qquad+ \varkappa_{ij}(y)\xi_j(s)\partial_i^x \bar{w}_1(s,x) 
+ Q_{ij}\varrho_k(y)\partial_{ijk}^x \bar{w}_1(s,x) +\mu_H\varrho_i(y)\partial^x_i\bar{w}_1(s,x),\nonumber
\end{align}
with $i,j,k,l= 1,2$. Except for the following three expressions, the terms in~\eqref{w4} were defined in Section~\ref{subsubsec:formal_expansion}:

\noindent $\Pi=(\Pi_{ijkl})\in H^1_\#(Y;\R^{16})$ with $\int_Y \Pi\dd{y}=0$ solves
\begin{align*}
-\diverg^y(a \nabla^y \Pi_{ijkl}) =  \partial_l^y(a\Lambda_{ijk}) +a\partial^y_l \Lambda_{ijk} + a\delta_{ij}\zeta_{kl} -R_{ijkl}-Q_{ij}\zeta_{kl};
\end{align*}
$\nu=(\nu_{ijk})\in H_\#^1(Y;\R^{8})$ with $\int_Y \nu\dd{y}=0$ solves
\begin{align*}
-\diverg^y(a \nabla^y \nu_{ijk}) = \partial^y_j( a\varkappa_{ik}) + a\partial^y_{j}\varkappa_{ik}-a\partial_k^y \zeta_{ij} 
- a\phi_i\delta_{jk} -S_{ijk};
\end{align*}
$\varrho=(\varrho_{i})\in H_\#^1(Y;\R^{2})$ with $\int_Y \varrho\dd{y}=0$ solves
\begin{align*}
-\diverg^y(a \nabla^y \varrho_{i}) =  \phi_{i}.
\end{align*}

\subsection{The general case with torsion}\label{appendix_A2} If $\tau\neq 0$, the coefficient of order four in the full 
asymptotic expansion~\eqref{ansatz:full_asymptotic_expansion} 
is $v_4=w_4\varphi_P + \bar{v}_4$. For $s\in[0,l]$, $x\in \omega$, and $y=x/\eps$ the term $w_4(s,x,y)$ is as in~\eqref{w4} and
\begin{align*}
 \bar{v}_4(s,x,y)=\tau^2(s) \bar{v}(s,x,y)\varphi_P(s) + \tau(s) \hat{v}(s,x,y)\varphi_P'(s) + \tau'(s)\check{v}(s,x,y)\varphi_P(s).
\end{align*}
Here,
\begin{align*}
 \bar{v}(s,x,y)&= \bigl(\zeta_{ij}(y)\delta_{kl}+\upsilon_{ijkl}(y) +\eta_{ijkl}(y)\bigr)(Rx)_j(Rx)_k\partial^x_{il}w_H(x) \\ &\qquad
+ \eta_{ijkl}(y)\partial_l^x\bigl((Rx)_j(Rx)_k\bigr)\partial_i^x w_H(x)+  2\vartheta_{ijk}(y)(\xi(s)\cdot x)(Rx)_j(Rx)_k\partial_i^xw_H(x)\\ 
& \qquad
+ \vartheta_{ijk}(Rx)_j(Rx)_k\partial_i^x\bar{w}_1(s,x),\\
 \hat{v}(s,x,y)&=  \bigl(\zeta_{ij}(y)+ \zeta_{ji}(y)\bigr)(Rx)_j\partial^x_i w_H(x) 
+ 2(\xi(s)\cdot x)(Rx)_i\phi_i(y)w_H(x) \\ 
&\qquad+ \partial_j^x(Rx)_i\zeta_{ij}(y)w_H(x)+ (Rx)_i\phi_i(y)\bar{w}_1(s,x), \\
 \check{v}(s,x,y)&= -\varkappa_{ij}(y)(Rx)_j\partial^x_{i}w_H(x),
\end{align*}
for $i,j,k,l=1,2$, with the definitions of Sections~\ref{subsubsec:formal_expansion} and~\ref{sec:propagation}, and the following two expressions:

\noindent $\eta=(\eta_{ijkl})\in H^1_\#(Y;\R^{16})$ with $\int_Y \eta\dd{y}=0$ solves
\begin{align*}
-\diverg^y(a \nabla^y \eta_{ijkl}) =  \partial_l^y(a\vartheta_{ijk}) +a\partial^y_l \vartheta_{ijk} - \int_{y}a\partial^y_l \vartheta_{ijk}\dd{y};
\end{align*}
$\upsilon=(\upsilon_{ijk})\in H_\#^1(Y;\R^{8})$ with $\int_Y \upsilon\dd{y}=0$ solves
\begin{align*}
-\diverg^y(a \nabla^y \upsilon_{ijkl}) = \partial^y_k( a\partial^y_j\zeta_{il}).
\end{align*}

\section{Boundary value estimate for the asymptotic expansion}\label{appendix_B}
Here we give a detailed proof of estimate~\eqref{est:boundary}, i.e.~
\begin{align*}
 \norm{w^\ast_\eps}_{H^{1/2}(\partial \omega)} \le C \eps^{1/2}
\end{align*}
for all $\eps>0$ sufficiently small. Recall that $w^\ast_\eps$ was defined in~\eqref{asymptotic_expansion}. 
The arguments follow closely along the lines of~\cite[Chapter~7.2]{CioDon99}. 

Let us define a family of smooth cut-off functions $\{m_\eps\}_\eps\subset C^\infty(\omega;[0,1])$ such that
$m_\eps=1$ if $\dist(x, \partial\omega)\leq \eps$, $\supp m_\eps\subset U_\eps:=\{x\in \omega\,:\, \dist(x, \partial\omega)< 2\eps\}$ and 
$\norm{\nabla m_\eps}_{L^\infty(\omega)}\leq C\eps^{-1}$. Due to $w_H=0$ on $\partial\omega$, the function 
\begin{align*}\varphi_\eps:=m_\eps (w_\eps^\ast-w_H)\in H^1(\omega)\end{align*} coincides on $\partial\omega$ in the sense of traces with
$w_\eps^\ast=w_\eps^\ast(s)=w_\eps^\ast(s, \cdot)$. Here again, we keep $s\in [0,l]$ fixed, but dispense with marking the dependence on $s$ in our notation. 
The trace theorem implies the existence of a constant $c>0$ such that
\begin{align}\label{estimate_CioDon}
\norm{w_\eps^\ast}_{H^{1/2}(\partial\omega)}= \norm{\varphi_\eps}_{H^{1/2}(\partial\omega)}\leq c\norm{\varphi_\eps}_{H^1(\omega)} 
=c\norm{\varphi_\eps}_{H^1(U_\eps)}.
\end{align}
Towards estimating $\varphi_{\eps}$ in the $H^1(U_\eps)$-norm, we observe first that
\begin{align}\label{L2varphi}
\norm{\varphi_\eps}_{L^2(U_\eps)} \leq C\eps.
\end{align}
For $\partial_j w_1$, $j=1, 2,$ with $w_1$ as in~\eqref{defw123} and $y=x/\eps$, one obtains
\begin{align*}
\partial_j w_1(x)=\frac{1}{\eps}\partial_j^y\phi_i(x/\eps)\partial_i^xw_H(x) +\phi_i(x/\eps)\partial_{ij}^x w_H(x) + \partial_j^x \bar{w}(x) 
+\xi_i\partial_j^x\hat{w}_i(x),\qquad x\in \omega,
\end{align*}
so that by the $C^\infty$-regularity of $w_H$, $\bar{w}$ and $\hat{w}$ and in view of the regularity assumption \textit{(H1)} on $\phi$, i.e.\ 
$\phi\in W_{\#}^{1, \infty}(Y;\R^2)$,
\begin{align*}
\norm{\nabla w_1}_{L^2(U_\eps;\R^2)}\leq \eps^{-1}\norm{\phi}_{W^{1,\infty}(Y;\R^2)} \norm{\nabla w_H}_{L^2(U_\eps)} +c\leq C\eps^{-1}\norm{w_H}_{H^1(U_\eps)}+c.
\end{align*}
For $\partial_k w_2$, $k=1,2,$ with $w_2$ as in~\eqref{defw123} and $y=x/\eps$, we have
\begin{align*}
\partial_k w_2(x)&=\frac{1}{\eps}\partial_k^y\zeta_{ij}(x/\eps)\partial_{ij}^xw_H(x) + \zeta_{ij}(x/\eps)\partial_{ijk}^xw_H(x) 
+ \frac{1}{\eps}\partial_k^x\phi_i(x/\eps)\partial_i^x\bar{w}_1(x) \\
&\qquad+\phi_i(x/\eps)\partial_{ik}^x\bar{w}_1(x)+\partial_k^x\bar{w}_2(x),\qquad x\in \omega,
\end{align*}
which in view of  $w_H, \bar{w}_1, \bar{w}_2\in C^\infty(\bar{\omega})$ and the fact that $\zeta$ and $\phi$ are $H^1_\#$-functions gives
\begin{align*}
\norm{\nabla w_2}_{L^2(U_\eps;\R^2)}\leq C\eps^{-1} +c.
\end{align*} 
For $w_n$, $n=3,4$, analogous arguments imply $\norm{\nabla w_n}_{L^2(U_\eps;\R^2)} \leq C\eps^{-1}+c$.
Summing up, 
\begin{align}\label{est20}
\norm{\nabla (w_\eps^\ast-w_H)}_{L^2(U_\eps;\R^2)} \leq c\norm{w_H}_{H^1(U_\eps)}+C\eps.
\end{align}
Moreover, from the structure of $w_\eps^\ast$ and the fact that $\abs{U_\eps}\leq \alpha \eps$ for some $\alpha>0$, depending only on $\omega$, we may conclude that, for $\eps>0$ small enough,
\begin{align}\label{est21}
\norm{w_\eps^\ast-w_H}_{L^2(U_\eps;\R^2)} &\leq \eps \norm{\phi}_{L^\infty(Y;\R^2)} \norm{\nabla w_H}_{L^2(U_\eps;\R^2)} + \eps\norm{\bar{w}_1}_{L^2(U_\eps)} +C\eps^2 \nonumber\\ &\leq c\eps\norm{w_H}_{H^1(U_\eps)}+C\eps^{3/2}. 
\end{align}
Joining~\eqref{L2varphi},~\eqref{est20} and~\eqref{est21}, while accounting for the properties of $m_\eps$,
we find
\begin{align}\label{est7}
\norm{\varphi_\eps}_{H^1(U_\eps)} \leq c\norm{w_H}_{H^1(U_\eps)}+C\eps^{1/2}.
\end{align}
On the other hand, the result of~\cite[Lemma~1.5]{Oleinik}, which is formulated for domains with smooth boundary, but extends to arbitrary 
Lipschitz domains (see~Remark~\ref{rem:Lipschitzboundary}), yields
\begin{align*} 
\norm{w_H}_{H^1(U_\eps)}\leq c\eps^{1/2}\norm{\nabla w_H}_{H^1(\omega;\R^2)}\leq c\eps^{1/2}
\end{align*}
for sufficiently small $\eps$. Together with~\eqref{estimate_CioDon} and~\eqref{est7}, this concludes the proof of estimate~\eqref{est:boundary}.

\begin{remark}\label{rem:Lipschitzboundary}
To be precise,~\cite[Lemma~1.5]{Oleinik} says that if $\omega$ is a bounded domain with a smooth boundary, then for all $\eps>0$ sufficiently small and every $w\in H^1(\omega)$,
\begin{align*}
\norm{w}_{L^2(U_\eps)}\leq c\eps^{1/2}\norm{w}_{H^1(\omega)}
\end{align*}
with a constant $c>0$ depending only on $\omega$. 
We would like to point out that the same statement still holds true, if $\omega$ is only a Lipschitz domain. 
To see this, we approximate $\partial \omega$ by $\Scal_\delta:=\partial \{x\in \omega:{\rm dist}(x, \partial \omega)< \delta\}$ for $\delta\to 0$. 
Since $\Scal_\delta$ can be viewed as the level 
sets of a suitable Lipschitz function provided $\delta$ is small enough, the trace theorem in conjunction with the coarea formula 
yields the asserted estimate. Indeed,
\begin{align*}
\norm{w}_{L^2(U_\eps)}^2=\int_0^{\eps}\left(\int_{\Scal_\delta}\abs{w}^2\dd{\Hcal^{1}}\right)\dd{\delta}\leq c\eps\norm{w}_{H^1(\omega)}^2
\end{align*}
for all $w\in H^1(\omega)$.
\end{remark}


\section*{Acknowledgments}
\sloppypar
The authors would like to express their thanks to Lu\'{i}s Trabucho for several stimulating conversations and to Christoph Kreisbeck
for his helpful comments on the physics background of the problem. 
C.K.\ was supported by the Funda\c c\~{a}o para a Ci\^{e}ncia e a Tecnologia 
through the ICTI CMU-Portugal Program in Applied Mathematics and UTA-CMU/MAT/0005/2009. Part of this research was carried out while C.K. was
at Universit\"at Regensburg. Traveling funds for C.K. provided by CMA are thankfully acknow\-ledged. 
L.M. was supported by the Funda\c c\~ao para a Ci\^{e}ncia e a Tecnologia,
through PEst-OE/MAT/UI0297/2011, PTDC/MAT109973/2009, and UTA-CMU/MAT/0005/2009.



\begin{thebibliography}{99}

\bibitem{Allaire08} Allaire, G., \textit{Periodic homogenization and effective mass theorems for the
              {S}chr\"odinger equation}, Lecture Notes in Math. \textbf{1946}, Springer-Verlag, Berlin (2008).
 
\bibitem{AllPia05} Allaire, G. and Piatnitski, A., Homogenization of the {S}chr\"odinger equation and effective mass theorems,
 \textit{Comm. Math. Phys.} \textbf{258}, 1 (2005), pp. 1--22.
 


\bibitem{BenLionsPapa78} Bensoussan, A., Lions, J.-L. and Papanicolaou, G., \textit{Asymptotic analysis for periodic structures},
Studies in Mathematics and its Applications \textbf{5}, North-Holland Publishing Co., Amsterdam (1978).

\bibitem{Bishop75} Bishop, R.~L., More than one way to frame a curve, \textit{Am.~Math.~Mon.} \textbf{82}, 3 (1975), pp. 246--251.

\bibitem{BoFr2010} Borisov, D. and Freitas, P., Asymptotics of Dirichlet eigenvalues and eigenfunctions of 
the Laplacian on thin domains in {$\mathbb{R}^d$}, \textit{J.~Funct.~Anal.} \textbf{258}, 3 (2010), pp. 893--912.

\bibitem{BMT07} Bouchitt{\'e}, G. and Mascarenhas, M.~L. and
              Trabucho, L., On the curvature and torsion effects in one dimensional
              waveguides, \textit{ESAIM Control Optim. Calc. Var.} \textbf{13}, 4 (2007), pp. 793--808.

\bibitem{BMT_Robin} Bouchitt{\'e}, G., Mascarenhas, M.~L. and Trabucho, L., Waveguides with Robin boundary conditions, \textit{J. Math. Phys.} \textbf{53}, 123517 (2012).

\bibitem{Cao2004} Cao, G., \textit{Nanostructures and Nanomaterials, Synthesis, Properties and Applications}, Imperial College Press, London (2004).

\bibitem{CDFK2005} Chenaud, B., Duclos, P., Freitas, P. and Krej{\v{c}}i{\v{r}}{\'i}k, D., Geometrically induced discrete spectrum in curved tubes, 
\textit{Differ.~Geometry Appl.} \textbf{23} (2005), pp. 95--105.

\bibitem{CioDon99} Cioranescu, D. and Donato, P., \textit{An introduction to homogenization}, Oxford Lecture Series in Mathematics and its Applications \textbf{17},
Clarendon Press Oxford University Press, New York (1999).

\bibitem{DalMaso_Gammaconvergence} Dal Maso, G., \textit{An introduction to {$\Gamma$}-convergence}, Progress in Nonlinear Differential Equations and their
              Applications \textbf{8}, Birkh\"auser Boston Inc., Boston (1993).
              
\bibitem{DautrayLions90} Dautray, R. and Lions, J.-L., \textit{Spectral theory and applications}, Mathematical analysis and numerical methods for science and
              technology \textbf{3}, Springer-Verlag, Berlin (1990).
              
\bibitem{DucEx95} Duclos, P. and Exner, P., Curvature-induced bound states in quantum waveguides in two
              and three dimensions, \textit{Rev. Math. Phys.} \textbf{7}, 1 (1995), pp. 73--102.
              

 \bibitem{FriedlanderSolomyak07} Friedlander, L. and Solomyak, M., On the spectrum of narrow periodic waveguides, \textit{Russ. J. Math. Phys.} \textbf{15}, 2 (2008), 
pp. 238--242.

\bibitem{FriedlanderSolomyak09} Friedlander, L. and Solomyak, M., On the spectrum of the {D}irichlet {L}aplacian in a narrow
              strip, \textit{Israel J. Math.} \textbf{170} (2009), pp. 337--354.
              
\bibitem{JikovKozlovOleinik94} Jikov, V.~V., Kozlov, S.~M. and Ole{\u\i}nik, O.~A., \textit{Homogenization of differential operators and integral
              functionals}, Springer-Verlag, Berlin (1994).
              
\bibitem{KesaI79} Kesavan, S., Homogenization of Elliptic Eigenvalue Problems: Part 1, \textit{Appl. Math. Optim.} \textbf{5} (1979), pp. 153--167.

\bibitem{KesaII79} Kesavan, S., Homogenization of Elliptic Eigenvalue Problems: Part 2, \textit{Appl. Math. Optim.} \textbf{5} (1979), pp. 197--216.

\bibitem{Krejcirik12} Krej{\v{c}}i{\v{r}}{\'i}k, D. and Sedivakova, H., The effective Hamiltonian in curved quantum waveguides under 
mild regularity assumptions, \textit{Rev. Math. Phys.} \textbf{24}, 1250018 (2012).

\bibitem{RoyalSoc} Lauhon, L.~J., Gudiksen, M.~S. and Lieber C.~M., 
Semiconductor nanowire heterostructures, \textit{Phil. Trans. R. Soc. Lond. A} \textbf{362} (2004), pp. 1247--1260.

\bibitem{Nature02} Lauhon, L.~J., Gudiksen, M.~S., Wang D. and Lieber C.~M., Epitaxial core-shell and core-multishell nanowire heterostructures, 
\textit{Nature} \textbf{420} (2002), pp. 57--61.


\bibitem{Oleinik} Ole{\u\i}nik, O.~A., Shamaev, A.~S.~and Yosifian, G.~A., \textit{Mathematical problems in elasticity and homogenization},
Studies in Mathematics and its Applications~\textbf{26}, North-Holland Publishing Co., Amsterdam (1992).
             
\bibitem{Nanoletter1} Pevzner, A., Engel, Y., Elnathan R., Tsukernik, A., Barkay, Z.
and Patolsky, F., Confinement-Guided Shaping of Semiconductor Nanowires and
Nanoribbons: ''Writing with Nanowires'', \textit{Nano Lett.} \textbf{12} (2012), pp. 7--12.  
              
\bibitem{MaterChem12} Qu, Y. and Duan, X., One-dimensional homogeneous and heterogeneous nanowires for solar energy conversion, 
\textit{J. Mater. Chem.} \textbf{22} (2012),
pp. 16171--16181.         


\bibitem{BullMaterSci07} Sarkar, J., Khan G.~G. and Basumallick A., Nanowires: properties, applications and synthesis via porous anodic
aluminium oxide template, \textit{Bull. Mater. Sci.} \textbf{30}, 3 (2007), pp. 271--290. 

\bibitem{Tang70}  Tang, C.~H., An Orthogonal Coordinate System for Curved Pipes, \textit{IEEE Transactions on 
Microwave Theory Techniques} \textbf{18} (1970), pp. 69--69. 


\bibitem{Vanninathan81} Vanninathan, M., Homogenization of eigenvalue problems in perforated domains, \textit{Proc. Indian Acad. Sci. Math. Sci.} \textbf{90}, 3
     (1981), pp. 239--271.

\bibitem{Nature11} Yan, H., Choe, H.~S., Nam, S., Hu, Y., Das, S., Klemic, J.~F., Ellenbogen, J.~C. and Lieber, C.~M., 
Programmable nanowire circuits for nanoprocessors, \textit{Nature} \textbf{470} (2011), pp. 240--244.

\end{thebibliography}
\end{document}